\documentclass[a4paper,10pt]{amsart}

\usepackage{setspace} 
\usepackage{amsthm} 
\allowdisplaybreaks
\usepackage{amsmath}
\usepackage{amssymb}
\usepackage{amscd}
\usepackage{mathrsfs} 
\usepackage{xcolor}
\usepackage[pagebackref,breaklinks,linktoc=all,hyperfootnotes=false]{hyperref} 
\hypersetup{
	colorlinks  =	true,
	linkcolor	=	{green!50!black},
	citecolor	=	{green!50!black},
	urlcolor    =   {green!50!black}
}
\usepackage{mathrsfs}
\usepackage[utf8x]{inputenc}
\usepackage{enumitem} 
\usepackage{ucs}
\usepackage{graphicx}
\usepackage{tabularx}
\usepackage{wrapfig}
\usepackage{pgfplots}
\usepackage{caption}
\usepackage{subfigure}
\usepackage{longtable}
\usepackage{multicol}
\usepackage{mathtools}
\usepackage{tikz-cd}
\usepackage{tikz}
\usepackage[margin=1in]{geometry}
\usepackage[skins]{tcolorbox}
\usepackage{minibox}
\usepackage{float} 	
\usepackage[all]{xy}
\usepackage[symbol]{footmisc} 

\makeatletter 
\def\@tocline#1#2#3#4#5#6#7{\relax
	\ifnum #1>\c@tocdepth 
	\else
	\par \addpenalty\@secpenalty\addvspace{#2}%
	\begingroup \hyphenpenalty\@M
	\@ifempty{#4}{%
		\@tempdima\csname r@tocindent\number#1\endcsname\relax
	}{%
		\@tempdima#4\relax
	}%
	\parindent\z@ \leftskip#3\relax \advance\leftskip\@tempdima\relax
	\rightskip\@pnumwidth plus4em \parfillskip-\@pnumwidth
	#5\leavevmode\hskip-\@tempdima
	\ifcase #1
	\or\or \hskip 1em \or \hskip 2em \else \hskip 3em \fi%
	#6\nobreak\relax
	\dotfill\hbox to\@pnumwidth{\@tocpagenum{#7}}\par
	\nobreak
	\endgroup
	\fi}
\makeatother

\newcommand{\z}{\mathbb{Z}}

\newcommand{\zp}{\mathbb{Z}_p}
\newcommand{\q}{\mathbb{Q}}

\newcommand{\cc}{\mathbb{C}}
\newcommand{\qp}{\mathbb{Q}_p}
\newcommand{\cp}{\mathbb{C}_p}

\newcommand{\A}{\mathbb{A}}
\newcommand{\ph}[1]{\phantom{#1}}
\newcommand{\mat}[4]{\left(\begin{smallmatrix}
		#1 & #2\\
		#3 & #4
	\end{smallmatrix}\right)}

\newcommand{\Addresses}{
	{
		\bigskip
		\footnotesize
		
		L. Dall'Ava, \textsc{Dipartimento di Matematica, Università degli Studi di Padova,
			Padova, Italy.}\par\nopagebreak
		\textit{E-mail address:} \href{mailto:luca.dallava@math.unipd.it}{\texttt{luca.dallava@math.unipd.it}}\par\nopagebreak
		\textit{URL:} \url{https://www.math.unipd.it/~dallava/}
		
	}
}

\newtheorem{theoremaleph}{Theorem}

\newtheorem{theorem}{Theorem}[subsection]
\newtheorem{definition}[theorem]{Definition}
\newtheorem{lemma}[theorem]{Lemma} 
\newtheorem{corollary}[theorem]{Corollary}
\newtheorem{proposition}[theorem]{Proposition}
\newtheorem{assumption}[theorem]{Assumption}
\newtheorem{remark}[theorem]{Remark}

\newtheorem{notation}[theorem]{Notation}
\newtheorem*{notation*}{Notation}

\numberwithin{equation}{section}

\author{Luca Dall'Ava}
\title{Hida theory for special orders}

\pgfplotsset{compat=1.16}

\begin{document}
	
	\begin{abstract}
		This note is devoted to the study of families of quaternionic modular forms arising from orders defined by Pizer. In this situation, the Hecke-eigenspaces are 2-dimensional contrary to the classical case of Eichler orders. The main result is a Control Theorem in the spirit of Hida, interpolating these 2-dimensional Hecke-eigenspaces. We restrict our attention to a definite rational quaternion algebra ramified at a single odd prime $\ell$.\newline
		
	\end{abstract}
	
	\maketitle
	
	{   \hypersetup{hidelinks}
		\tableofcontents}

	\section{Introduction and statement of the result}
	
	Let $\ell$ be an odd prime, $N\geq 1$ an integer prime to $\ell$ and $k\geq 2$ an integer. Let $B$ be the unique, up to isomorphism, quaternion algebra over $\q$ ramified exactly at $\ell$ and $\infty$. Take $R$ to be an Eichler order of level $N$ in $B$ and consider the space of $\mathbb{C}$-valued cuspidal quaternionic newforms with level $R$, denoted by $\mathscr{S}_k^{new}(R,\mathbb{C})$; we recall its precise definition in Section \ref{Quaternionic Eisenstein series and newforms} but, roughly speaking, it represents the space of non-Eisenstein quaternionic modular forms which do not satisfy a lower level structure. The Jacquet--Langlands correspondence ensures an injective transfer between automorphic representations of the algebraic group associated with $B^\times$ and $GL_2/\q$, but at the level of automorphic forms it takes an explicit realization, often referred to as the Eichler--Jacquet--Langlands correspondence. More precisely, one obtains the Hecke-equivariant isomorphism $\mathscr{S}_k^{new}(R,\mathbb{C})\cong S_k^{new}(\Gamma_0(N\ell),\mathbb{C})$. On the other hand, in order to study modular forms with higher level structure at $\ell$, one needs more general orders, namely the \emph{special orders} defined by Pizer and Hijikta--Pizer--Shemanske. As the precise definition of special orders does not contribute necessarily to the understanding of this introduction, we prefer to avoid technicalities and postpone it to Definition \ref{def special orders}. In the case where $R$ is such a special order in $B$, corresponding to level structure $N\ell^{2r}$ with $r\geq 1$, the relation between automorphic forms changes considerably. Before stating the precise relation, in order to be consistent with the notation present in \cite{HPS1989basis}, let us introduce the following (unfortunately unconventional but rather useful) piece of notation. Throughout the whole document, for any module $M$, we write $2M$ for the direct sum $M\oplus M$. The Jacquet--Langlands correspondence, for $r\geq 2$, takes the explicit form of the following Hecke-isomorphism
	\begin{equation}\label{iso hps intro}
		2 S_k^{new}(\Gamma_1(N\ell^{2r}),\mathbb{C}) \cong \mathscr{S}^{new}_k(R,\mathbb{C})\oplus \bigoplus_{\chi}2 S_k^{new}(\Gamma_1(N\ell^{r}),\chi^2,\mathbb{C})^{\otimes\overline{\chi}},
	\end{equation}
	where two copies of $S_k^{new}(\Gamma_1(N\ell^{2r}),\mathbb{C})$ must be taken into account and the sum of the twisted spaces $S_k^{new}(\Gamma_1(N\ell^{r}),\chi^2,\mathbb{C})^{\otimes\overline{\chi}}$ (see beginning of Section \ref{quaternionic lifts of modular forms}) runs over certain primitive characters modulo $\ell^r$. Equation (\ref{iso hps intro}) has a slightly more complicated expression for $r=1$, but the above situation is already explanatory of the general phenomenon; an exhaustive statement can be found in Theorem \ref{Thm 7.16-7}. For any choice of an isomorphism in Eq. (\ref{iso hps intro}) we see that any classical newform of level $N\ell^{2r}$, which is twist-minimal at $\ell$ (as in Definition \ref{twist min}), lifts to two linearly independent quaternionic newforms with the same Hecke-eigenvalues away from the level. The above-mentioned situation has been extensively studied in a series of works by H. Hijikata, A. Pizer, and T. Shemanske, most notably \cite{Pizer80p2} and \cite{HPS1989basis}. In the light of this multiplicity, it is natural to ask whether these quaternionic modular forms still live in $p$-adic families, for $p$ an odd prime different from $\ell$ and prime to $N$. In this note we provide a positive answer to this question. 
	
	As introduced above, let $\ell$, $p$ and $N$ be, in the order, two distinct odd primes and a positive integer prime to both $\ell$ and $p$. For the sake of simplicity, we restrict here to the case of trivial character at $\ell$ and exponent $2r\geq 4$; we refer the reader to Theorem \ref{Control theorem for special orders} for the general statement. As in Definition 2.4 of \cite{GreenbergStevens1993}, we consider $\mathcal{R}$ to be the \emph{universal ordinary $p$-adic Hecke algebra} of tame level $N\ell^{2r}$. Considering the Iwasawa algebra $\zp[\![ \zp^\times ]\!]$, $\mathcal{R}$ represents the $\zp[\![ \zp^\times ]\!]$-algebra of Hecke-operators acting on Hida families of tame level $N\ell^{2r}$. For any continuous group homomorphism $\kappa:\mathcal{R}\longrightarrow \overline{\qp}$, we say that $\kappa$ is an arithmetic homomorphism if its restriction to $1+p\zp \subseteq \zp[\![ \zp^\times ]\!]$ defines a character of the form $z\longmapsto z^{k-2}\varepsilon(z)$, for $k\in \z_{\geq 2}$ and $\varepsilon$ a $p$-adic character of conductor $p^n$, $n\geq 0$. We recall that these homomorphisms are identified with points on the so-called weight space. As usual, we associate to each arithmetic homomorphism $\kappa$ the couple $(k,\varepsilon)$. For any $\kappa$ we also denote its kernel by $\mathcal{P}_\kappa$ and the corresponding localization of $\mathcal{R}$ by $\mathcal{R}_{\mathcal{P}_\kappa}$. Let $f$ be a classical modular newform in $S_k(\Gamma_0(Np\ell^{2r}),\mathbb{C})$ and assume that $f$ is twist-minimal at $\ell$. Moreover, if $f$ is $p$-ordinary, we can consider the unique Hida family $f_\infty$ associated with $f$ by the works of Hida and Wiles. For each arithmetic homomorphism $\kappa$ we denote by $f_\kappa$ the specialization of $f_\infty$ at $\kappa$. We set $F$ (resp. $F_\kappa$) to be the field extension of $\qp$ generated by the Fourier coefficients of $f$ (resp. $f_\kappa$) and take $\mathcal{O}$ (resp. $\mathcal{O}_\kappa$) to be its ring of integers. Fix now $R$ to be a maximal order in the quaternion algebra $B$ which contains the family of nested orders $\{R^n\}$,
	\begin{equation}
		\cdots \subset R^{n+1}  \subset R^n\subset \cdots R^0\subset R, \,\,\,\,\,\textrm{ $R^n$ is a special order of level $Np^n\ell^{2r}$.}
	\end{equation}
	At all places $q\neq\ell,\infty$, we fix isomorphisms $\iota_q: B\otimes_{\q}\q_q\cong M_2(\q_q)$ such that $R^{n}\otimes_{\z} \z_q$ is identified with the upper triangular matrices modulo $Np^{n}$. For each $n$ we consider the compact open subgroup $U_n\subset \widehat{R^n}:=R^n\otimes_{\z} \widehat{\z}$ defined as
	\begin{equation}
		U_n=\left\{g=(g_q)\in \widehat{R}^{n\times}\mid\iota_q(g_q) \equiv \left(\begin{smallmatrix}
			* & * \\
			0 & 1
		\end{smallmatrix}\right)\pmod{Np^n\z_q},\textrm{ for }q\mid Np^n\right\}.
	\end{equation}
	Let $V_{k-2}(\mathcal{O}_\kappa)$ be the dual of $L_{k-2}(\mathcal{O}_\kappa)$, the space of homogeneous polynomials in $\mathcal{O}_\kappa[X,Y]$ of degree $k-2$, endowed with the action $|_{u_p}$ of $GL_2(\zp)$, induced by the left multiplication $\mat{a}{b}{c}{d}(X,Y)^t=(aX+bY,cX+dY)^t$. Denoting $\widehat{B}=B\otimes \A_{\q,f}$ for $\A_{\q,f}$ the finite ad\`{e}les of $\q$, we consider, as in Definition \ref{quaternionic mf k def}, the space of quaternionic $p$-adic modular forms of weight $k\geq 2$, character $\varepsilon$ (of conductor $p^n$) and level $U_n$,
	\begin{multline}
		S_k(U_n,\varepsilon,\mathcal{O}_\kappa):=\Big\{\varphi:\widehat{B}^\times\rightarrow V_{k-2}(\mathcal{O}_\kappa)\mid \varphi(b\tilde{b}uz)=\varepsilon(z)z_p^{k-2}\varphi(\tilde{b})|_{u_p},\\ 
		\text{ for } b\in B^\times,\tilde{b}\in \widehat{B}^\times, u\in U_n,z\in \A_{\q,f}^\times \Big\}.
	\end{multline}
	More generally, let $\mathsf{X}$ be the subset of \emph{primitive vectors} in $\zp^2$, namely the subset of vectors with at least one component which is not divisible by $p$, and consider the space $\mathcal{M}(\mathsf{X},\mathcal{O})$ of $\mathcal{O}$-valued measures on $\mathsf{X}$. We construct, following Definition \ref{measure valued quat forms}, the space of measure-valued quaternionic modular forms
	\begin{multline}
		S_2(U_0, \mathcal{M}(\mathsf{X},\mathcal{O})):=
		\Big\{\varphi:\widehat{B}^\times\rightarrow \mathcal{M}(\mathsf{X},\mathcal{O})\mid \varphi(b\tilde{b}uz)=\varphi(\tilde{b})|_{u_p},\\
		\text{ for } b\in B^\times,\ph{.}\tilde{b}\in \widehat{B}^\times,\ph{.} u\in U_0,\ph{.}z\in \A_{\q,f}^\times \Big\},
	\end{multline}
	for $|_{u_p}$ the action of $GL_2(\zp)$ induced by the left multiplication on the variables. By integration, we induce, for any arithmetic homomorphism $\kappa=(k,\varepsilon)$, a specialization map
	\begin{equation}
		\nu_\kappa: S_2(U_0, \mathcal{M}(\mathsf{X},\mathcal{O}))\longrightarrow S_k(U_n,\varepsilon, \mathcal{O})
	\end{equation}
	such that
	\begin{equation}
		\nu_{\kappa}(\varphi)(\tilde{b})(P):=\int_{p\zp\times \zp^\times}\varepsilon_{\A}^{-1}(y)P(x,y)d(\varphi(\tilde{b}))(x,y)
	\end{equation}
	for $\phi$ and $\tilde{b}$ as above, and any $P\in L_{k-2}(\mathcal{O}_\kappa)$; all details can be found in Section \ref{specialization maps}. Considering the ordinary component of $S_2(U_0,\varepsilon, \mathcal{M}(\mathsf{X},\mathcal{O}))$, which we denote by $\mathbb{W}$, the specialization maps descend to maps between the ordinary components
	\begin{equation}
		\nu_{\kappa}^{ord}: \mathbb{W}\longrightarrow S_k(U_n,\varepsilon, \mathcal{O})^{ord},
	\end{equation}
	for $S_k(U_n,\varepsilon, \mathcal{O})^{ord}$ the subspace of $p$-ordinary quaternionic forms in $S_k(U_n,\varepsilon, \mathcal{O})$. As the algebra $\mathcal{R}$ acts on the space of Hida families, we can consider the $f_\infty$-isotypic component
	\begin{equation}
		\left(\mathbb{W}\otimes_{\mathcal{O}[\![ 1+p\zp ]\!]} \mathcal{R}\right)[f_\infty],
	\end{equation}
	that is, the component of $\mathbb{W}\otimes_{\mathcal{O}[\![ 1+p\zp ]\!]} \mathcal{R}$ where the Hecke-operators act with the same $\mathcal{R}$-eigenvalues of $f_\infty$. Up to a mild condition on the level $N\ell^{2r}$, as explained in Remark \ref{freeness}, one can assume the $\mathcal{R}$-module $\mathbb{W}\otimes_{\mathcal{O}[\![ 1+p\zp ]\!]} \mathcal{R}$ to be free, as we do here. We can hence state our main result under the above simplifying restrictions for the character and the power of $\ell$; the general statements are the content of Theorems \ref{Control theorem for special orders} and \ref{Control theorem for special orders 2}.
	\begin{theoremaleph}[Control theorem for special orders]\label{Control theorem for special orders Intro}
		With the above notation, suppose that $f$ is twist-minimal at $\ell$. For any arithmetic homomorphism $\kappa:\mathcal{R}\longrightarrow \overline{\qp}$, the map $\nu_\kappa$ (of Proposition \ref{Prop inj}, induced by the specialization map) induces an isomorphism of $2$-dimensional $F_\kappa$-vector spaces
		\begin{equation}
			\left(\left(\mathbb{W}\otimes_{\mathcal{O}[\![ 1+p\zp ]\!]} \mathcal{R}\right)[f_\infty]\right)\otimes_{\mathcal{R}}\mathcal{R}_{\mathcal{P}_{\kappa}}/\mathcal{P}_\kappa\mathcal{R}_{\mathcal{P}_{\kappa}} \overset{\cong}{\longrightarrow} \left(S_k(U_n,\varepsilon, F_\kappa)^{ord}\right)[f_\kappa].
		\end{equation}
	\end{theoremaleph}
	The two main ingredients needed for the proof of Theorem \ref{Control theorem for special orders Intro} are the above isomorphism (\ref{iso hps intro}) and the seminal paper \cite{Hida1988b}. The results proved in \cite{Hida1988b} for definite quaternion algebras over totally real fields different from $\q$ remain true in the case of definite quaternion algebras over $\q$, as already noticed in Section 3 of \cite{LongoVigni2012} and Section 4 of \cite{Hsieh2021}, and in the case of special orders, as remarked in Remarks \ref{freeness} and \ref{remark on hida}. The strategy of the proof is then a generalization of the work \cite{LongoVigni2012}, from which we take inspiration.
	
	Theorem \ref{Control theorem for special orders Intro} extends the foundational results of Hida theory to the case of quaternionic modular forms with special level structure, allowing to consider quaternionic $p$-adic families with tame level $\ell^{2r}$ over the quaternion algebra $B$, which, we remark, is ramified at $\ell$. We highlight again that the situation discussed here differs markedly from the classical case of Eichler orders, where all the local non-archimedean automorphic representations at $\ell$ are 1-dimensional.  In particular, the novelty of this result lies in the rank of the Hecke-eigenspaces being  and no more 1 as in the classical Eichler case.\\ 
	
	An additional motivation for such a Control Theorem originates from the desire to study points outside the interpolation region of the triple product $p$-adic $L$-function, as treated in \cite{DallAva2021Approx-p}. More precisely, the present study together with \cite{DallAva2021Approx-p} are motivated by a conjecture of Bertolini--Seveso--Venerucci and by the wish to provide computational support to it.
	
	The present work leaves open some questions which we briefly discuss in Section \ref{A small remark on related works and open questions}. The main one is whether it is possible to distinguish the 2-dimensional spaces of quaternionic modular forms attached to special orders, first in the case of a classical eigenspace, and then for families.\\
	
	\textbf{Acknowledgments:} This note presents the content of a section of the author's doctoral dissertation \cite{DallAva2021PhD}; he expresses his gratitude to his supervisor Massimo Bertolini. Many thanks go also to Matteo Longo for several helpful discussions, among the others on \cite{LongoVigni2012} and \cite{Hida1988b}. The author is grateful to the anonymous referee for the valuable comments, suggestions, and questions which led to a significant improvement of the exposition. The author also wishes to thank Matteo Tamiozzo, for numerous mathematical conversations, and Jonas Franzel, for reading an early draft of this work and finding out various typos. The author is grateful to the Universit\"at Duisburg--Essen, where the main part of this work has been carried out, as well as to the Università degli Studi di Padova (Research Grant funded by PRIN 2017 ``Geometric, algebraic and analytic methods in arithmetic'') for their financial support.
	
	\section{Quaternionic orders and modular forms}\label{quaternionic orders and modular forms}
	
	We begin by recalling the definitions of the various quaternionic orders as well as the definition of quaternionic modular forms, both $p$-adic and classical, for a definite quaternion algebra over $\q$. Special orders are a generalization of the classical Eichler orders, which are needed for studying both higher ramification and the presence of a character at primes where the quaternion algebra is ramified. We refer the reader to \cite{HPS1989basis} and \cite{HPS1989orders} for all the details that we are not recalling here.
	
	We fix, once and for all, a choice of field embeddings $\overline{\q}\hookrightarrow\overline{\qp}\hookrightarrow \cc$. For any prime $q$ we denote $q$-adic valuation by $\nu_q$, normalized so that $\nu_q(q)=1$.
	
	\subsection{Special orders}\label{special order}
	
	Let $B$ the unique, up to isomorphism, quaternion algebra over $\q$ with discriminant $D$. Fix an isomorphism $\iota_q:B_q:=B\otimes \q_q\cong M_2(\q_q)$ for each $q\nmid D$. We denote the reduced norm of $b\in B$ by $n(b)\in\q$. Let $q$ be an odd rational prime, and fix $u\in\z$ to be a quadratic non residue modulo $q$. The local field $\q_q$ has a unique quadratic unramified extension $\q_q(\sqrt{u})$ and two quadratic ramified ones, $\q_q(\sqrt{q})$ and $\q_q(\sqrt{uq})$. For $L_q$ one of these quadratic extensions, we denote by $\mathcal{O}_{L_q}$ its ring of integers. Set
	\begin{equation}
		M_2^0(N\z_q):=\left\{\gamma\in M_2(\z_q) \mid  \gamma\equiv \mat{*}{*}{0}{*} \pmod{N\z_q}\right\}=M_2^0(q^{\nu_q(N)}\z_q),
	\end{equation}
	and, for $r\geq 1$,  
	\begin{equation}
		M(L_q,r):=\mathcal{O}_{L_q}+\left\{x\in B_q\mid n(x)\in q\z_q \right\}^{r-1}
		=\mathcal{O}_{L_q}+\left\{x\in 	B_q\mid n(x)\in q^{r-1}\z_q \right\}.
	\end{equation}
	We notice that, for $r=1$, $	M(L_q,1)$ is the unique maximal ideal of $B_q$.
	\begin{definition}[\cite{HPS1989orders}, Def. 6.1]\label{def special orders}
		An order $R$ in $B$ is said to be a special order of level $M=N\cdot\prod_{\ell\mid D}\ell^{m_\ell}$ if
		\begin{enumerate}[label=(\roman*),font=\itshape]
			\item $R_q:=R\otimes \z_q$ is conjugate to $M_2^0(N\z_q)$ by an element of $B_q^\times$ (via $\iota_q$), for each $q\nmid D\infty$;
			\item there exists a quadratic extension $L_\ell$ of $\q_\ell$ such that $R_\ell$ is conjugate to $M(L_\ell,m_\ell)$, for each $\ell|D$.
		\end{enumerate}
	\end{definition} 
	\noindent In the following, we choose for each $q\nmid D$ the isomorphism $\iota_q$ such that $\iota_l(R_q)=M_2^0(N\z_q)$. If in the above definition we take the level $M$ to be such that $D||M$, we obtain the usual definition of an Eichler order of level $N$ (see \cite{Pizer1980}, page 344).\newline
	
	From now on, we assume $D$ to be odd and we fix a particular choice of special orders and quadratic field extensions. We follow the thorough summary given in Section $2.2$ of \cite{LRdvP2018}, which is based on a careful analysis of \cite{HPS1989basis} and takes into account more general orders. Let $f$ be any newform in $ S_2(\Gamma_1(N \prod_{\ell\mid D}\ell^{e_\ell}), \mathbb{C})$. In order to be able to lift $f$ to a quaternionic modular form, we fix the choice of the special order $R$ such that the quadratic extensions $L_{\ell}$ of $\q_{\ell}$ and the exponents $m_{\ell}$ are as follows. For any (odd) prime $\ell\mid D$, if
	\begin{enumerate}
		\item $e_\ell$ is odd, we take $L_\ell$ to be the unramified extension of $\q_{\ell}$ and $m_{\ell}=e_\ell$;
		\item $e_\ell$ is even, we take $L_\ell$ to be one of the two ramified extension of $\q_{\ell}$ and $m_{\ell}=e_\ell$.
	\end{enumerate}
	The case of even discriminant presents some further difficulties and, as our main case of interest is the case of $\ell$ odd, we omit it and refer the reader to \cite{HPS1989basis} and \cite{LRdvP2018}.\newline
	
	We recall a part of the notation already used in the Introduction.
	\begin{notation}
		Let $\widehat{\z}$ be the profinite completion of $\z$. We set $\widehat{B}=B\otimes_\q \A_{\q,f}$ and $\widehat{R}=R\otimes_\z \widehat{\z}$, where $\A_{\q,f}=\q \widehat{\z}$ are the finite ad\`{e}les of $\q$.
	\end{notation}
	Special orders satisfy properties similar to the Eichler ones.
	\begin{lemma}\label{open compact order}
		$\widehat{R}^\times$ is a compact open subgroup of $\widehat{B}^\times$. In fact, this is true for each one of its local components.
	\end{lemma}
	\begin{proof}
		This lemma is a classical result whenever $R$ is an Eichler order (see \emph{e.g.} Sections 5.1 and 5.2 of \cite{Miyake2006}). We consider the case of a special order. Lemma 5.1.1 of \cite{Miyake2006} tells us that $R_q^\times$ is compact in $B_q^\times$, independently of the order. By definition of special order, it is enough to consider $M(L_\ell,r)$. Since the reduced norm is continuous, $\{x\in 	B_\ell\mid n(x)\in \ell^{r-1}\z_\ell \}$ is open and thus, as the sum is a continuous homomorphism, we deduce the claim.
	\end{proof}
	\begin{proposition}[\cite{P1977}, Proposition 2.13]
		All special orders have finite class number. Moreover, it depends only on the level and not on the specific choice of the special order.
	\end{proposition}
	\begin{lemma}[\cite{HPS1989orders}, Lemma 7.4]\label{HPS1989orders Lemma 7.4}
		Let $R$ be a special order of level $N$. Then there exists a set of ideal class representatives $\{I_1,\ldots,I_h\}$ for the left $R$-ideal classes, such that $I_i\otimes\z_q=R_q$ for all $q$ dividing the level.
	\end{lemma}
	
	
	\subsection{Characters}\label{lift character}
	
	Let $R$ be a special order of level $M=N\cdot\prod_{\ell|D}\ell^{m_\ell}$ and let $\chi$ be a Dirichlet character with conductor $C$.
	\begin{assumption}\label{chi-R}
		Assume that $\nu_q(N)\geq \nu_q(C)$ for all $q\mid N$ and, for each $\ell\mid D$, that
		\begin{equation}
		    \begin{aligned}
			m_\ell\geq
    			\begin{cases}
    				2\, \nu_\ell(C)-1 & \textrm{if $L_\ell$ is unramified over $\q_\ell$,}\\
    				2\, \nu_\ell(C) & \textrm{if $L_\ell$ is ramified over $\q_\ell$.}
    			\end{cases}  
    		\end{aligned}
		\end{equation}
	\end{assumption}
	\noindent The choice of character in Assumption \ref{chi-R} is motivated by the explicit form of the Jacquet--Langlands correspondence recalled in the Introduction (see Section \ref{quaternionic lifts of modular forms} for details).
	
	We want to extend $\chi$ to a character $\widetilde{\chi}$ of $\widehat{R}^\times$ and for this purpose we must deal with several sub-cases. First of all, we decompose $\chi=\prod_{q\mid C}\chi_q$ by the Chinese Reminder Theorem and we define each $\widetilde{\chi_q}$ as follows.
	\begin{enumerate}
		\item	If $q\mid M$ and $q\nmid C$, we set $\widetilde{\chi_q}(\alpha)=1$ for each $\alpha\in R_q$.
		\item If $q\mid N$ and $q\mid C$, we set $\widetilde{\chi_q}(\alpha)=\chi_q(d)$ for $\alpha=\left(\begin{smallmatrix}
			a & b\\
			c & d
		\end{smallmatrix}\right)\in R_q=M_2^0(N\z_q)$.
		\item If $q$ is odd, $q\mid M/N$ and $q\mid C$ we deal with three further sub-cases:
		\begin{enumerate}[label=(\roman*),font=\itshape]
			\item	If $q\mid\mid C$ and $\chi_q$ is odd, we can always find a lift to $\mathcal{O}_{L_q}/\mathfrak{m}_{L_q}\supseteq \z_q/q\z_q$ and thus to $\mathcal{O}_{L_q}$, which we call $\chi_{L_q}$; here $\mathfrak{m}_{L_q}$ is the maximal ideal of $\mathcal{O}_{L_q}$. Because of Assumption \ref{chi-R}, $R_q=M(L_q,m_q)$ is contained in $M(L_q,2 \nu_q(C)-1)$ (if $L_q$ is unramified) or in $M(L_q,2 \nu_q(C))$, hence we set $\widetilde{\chi_q}(\alpha+\beta)=\chi_{L_q}(\alpha)$ for each $\alpha+\beta\in R_q=M(L_q,m_q)$ (and $\alpha\in \mathcal{O}_{L_q}$).
			\item	If $q^e\mid \mid C$, with $e\geq 1$, and $\chi_q$ is even, we can always find a character $\psi$ such that $\psi^2=\chi_q$ and with conductor $cond(\psi)=cond(\chi_q)$. As remarked in Section 7.2 of \cite{HPS1989orders}, the choice of this character is not important, but the fact that a particular choice is fixed is. We set $\widetilde{\chi_q}(\alpha)=\psi(n(\alpha))$.
			\item	If $q^e\mid \mid C$, with $e>1$, and $\chi_q$ is odd, we write $\chi_q=\varepsilon\cdot \phi$ for a fixed choice of characters $\varepsilon$ odd and with $cond(\varepsilon)=q$, and $\phi$ even. Thus, proceeding analogously to the previous sub-cases, we set $\widetilde{\chi_q}=\widetilde{\varepsilon}\cdot \widetilde{\phi}$.
		\end{enumerate}
		\item If $q=2$, $2\mid N/M$ and $2\mid C$, one proceeds in a similar way as in case $3.a$.
	\end{enumerate}
	Patching together the local lifts, we define
	\begin{equation}
		\widetilde{\chi}(b):=\prod_{q\mid N}\widetilde{\chi_q}(b_{q}),
	\end{equation}
	for any $b\in \widehat{R}^\times$.
    In particular, if $I$ is a lattice in $B$ such that $I_q=I\otimes\z_q=R_q$ for each $q|N$, and $b\in I$, we have
	\begin{equation}
		\widetilde{\chi}(b)=\prod_{q\mid N}\widetilde{\chi_q}(b).
	\end{equation}
	We refer to \cite{HPS1989orders}, Section 7.2 for all the details.
	
	\subsection{Quaternionic modular forms of weight $2$}\label{Quaternionic modular forms of weight $2$}\nocite{Buzzard2004}\nocite{BDI2010}
	
	Take $B$ as in the above Section \ref{special order}, with $R\subset B$ a special order of level $M$. Recall that we fixed isomorphisms $\iota_q:B_q:=B\otimes \q_q\cong M_2(\q_q)$ for each $q\nmid D$, such that $\iota_q:R_q :=R\otimes \z_q\cong M_2^0(M\z_q)$. Set $B(\A_\q)^\times=(B\otimes_\q\A_\q)^\times$ and 
	\begin{equation}
	    R(\A_\q)^\times=\{r\in B(\A_\q)^\times\mid (r_q)_{q<\infty}\in \widehat{R}^\times\}.
	\end{equation}
	We extend any character $\widetilde{\chi}$ as in the above Section \ref{lift character} to $R(\mathbb{A}_{\q})^\times$ imposing $\widetilde{\chi}(b)=\widetilde{\chi}(b_f)$, where $b_f\in \widehat{R}^\times$ is the finite part of $b\in R(\mathbb{A}_{\q})^\times$.
	\begin{definition}\label{quaternionic mf def}
		We define the space of weight-$2$ quaternionic modular forms with level structure $R(\A_\q)^\times$, character $\chi$ satisfying Assumption \ref{chi-R}  and $\cc$-coefficients, as the $\cc$-vector space $S_2(R, \widetilde{\chi})$ of all continuous functions $\varphi:B(\A_\q)^\times\longrightarrow\cc$ satisfying
		\begin{equation}
			\varphi(b \tilde{b} r)=\widetilde{\chi}^{-1}(r)\varphi(\tilde{b}),
		\end{equation}
		for all $b\in B^\times$, $\tilde{b}\in B(\A_\q)^\times$ and $r\in R(\A_\q)^\times$.
	\end{definition}
	As in Chapter 5 of \cite{HPS1989basis}, we can decompose $B(\A_\q)^\times$ as a finite union of distinct double cosets
	\begin{equation}
		B(\A_\q)^\times=\coprod_{i=1}^{h} B^\times x_i R(\A_\q)^\times
	\end{equation}
	where $h=h(R)$ is the class number of $R$. Since $B$ is definite, the analogous decomposition holds for $\widehat{B}^\times$, namely $\widehat{B}^\times=\coprod_{i=1}^{h}B^\times \widehat{x_i} \widehat{R}^\times$, with $\widehat{x_i}=(x_{i,q})_{q<\infty}$. By the above Lemma \ref{HPS1989orders Lemma 7.4}, the representatives $x_i=(x_{i,q})_{q}\in B(\A_\q)^\times$ can be taken to lie in $R(\A_\q)^\times$, in particular $x_{i,q}\in R_q^\times$ for each prime $q|M$. If we fix the representatives in this way, we have an explicit description of quaternionic modular forms. By the definition of a quaternionic modular forms and the double coset decomposition, a quaternionic modular form $\varphi$ is uniquely determined by its values on the representatives. More precisely, for $i=1,\ldots, h$, let $\widetilde{\Gamma}_{x_i}:= B^\times \cap x_i^{-1}R^\times x_i$ and define
	\begin{equation}
		\cc_{\widetilde{\chi},i}:=\left\{c\in\cc\mid \widetilde{\chi}(\gamma)\cdot c=c,\textrm{ for each }\gamma\in \widetilde{\Gamma}_{x_i}\right\}.
	\end{equation}
	As thoroughly explained in \emph{loc.cit.}, the above observations yield the identification
	\begin{equation}
		S_2(R, \widetilde{\chi})\cong \bigoplus_{i=1}^{h}\cc_{\widetilde{\chi},i},
	\end{equation}
	given by $\varphi\longmapsto (\varphi(x_1),\ldots,\varphi(x_h))$. We are allowed to consider different coefficients, in fact the above identification still holds when we replace $\cc$ by $\q(\widetilde{\chi})$, the field extension of $\q$ generated by the values of the character $\widetilde{\chi}$. By extension of scalars we recover $S_2(R, \widetilde{\chi})=S_2(R, \widetilde{\chi} ;\q(\chi))\otimes\cc$ and we can consider $p$-adic coefficients $S_2(R, \widetilde{\chi} ;\overline{\qp})=S_2(R, \widetilde{\chi} ;\q(\chi))\otimes\overline{\qp}$, for $p$ a prime which does not divide the reduced discriminant of $B$.
	\begin{remark}\label{suitable choices}
		All the above constructions and definitions are, up to isomorphism, independent of the specific choice of the special order. Moreover, fixing compatible choices of the lifting characters $\widetilde{\chi}$, all the constructions are compatible with respect to the inclusion of special orders.
	\end{remark}
	We end this section with the following fact: often the groups $\widetilde{\Gamma}_{x_i}$ have cardinality $2$, \emph{i.e.} $\widetilde{\Gamma}_{x_i}=\{\pm 1 \}$.
	\begin{proposition}[\cite{Pizer80p2}, Proposition 5.12]\label{torsion free Gammas}
		Let $R$ be a special order of level $M\ell^2$ in the quaternion algebra over $\q$ ramified exactly at $\ell$ and $\infty$. Then
		\begin{equation}
			\# R^\times =\begin{cases}
				2 & \textrm{ if }\ell>3,\\
				\textrm{either $2$ or $6$} & \textrm{ if }\ell=3.
			\end{cases}
		\end{equation}
		Moreover, if $\ell=3$ and $2|M$ or $M$ is divisible by a prime $q\equiv 2 \pmod{3}$, then $\# R^\times=2$.
	\end{proposition}
	
	\subsection{$p$-adic quaternionic modular forms for special orders}\label{$p$-adic quaternionic modular forms for special orders}
	
	Let $p$ and $\ell$ be two distinct rational odd primes. From now on, we denote by $B$ the (unique up to isomorphism) definite quaternion algebra over $\q$ with discriminant $\ell$ and we fix $R$ to be a maximal order in $B$. For $N$ a fixed positive integer, prime to both $p$ and $\ell$, consider a family of nested special orders $\{R^{n}\}_{n\geq 0}$ satisfying
	\begin{equation}
		\cdots \subset R^{n+1}  \subset R^n\subset \cdots R^0\subset R, \,\,\,\,\,\textrm{ $R^n$ is a special order of level $Np^n\ell^{2r}$,}
	\end{equation}
	where $r\geq 1$. Up to conjugation, we can suppose that the orders $R_{n}$ are all canonical orders of level $Np^n\ell^{2r}$, that is, as in Definition \ref{def special orders}. For any prime $q$ different from $\ell$, we assume that the fixed isomorphism $\iota_q:B_q\cong M_2(\q_q)$ satisfies $\iota_q R_q^\times \cong GL_2(\z_q)$ and
	\begin{equation}
	    \iota_q(R^n_q)^\times\cong \Gamma_0(Np^n\z_q):=M_2^{0}(Np^n\z_q)\cap GL_2(\z_q).
	\end{equation}
	\begin{definition}\label{U1(R)}
		We define (\emph{cf.} Lemma \ref{open compact order}) open compact subgroups $U_n\subset \widehat{B}^\times$,
		\begin{equation}
			U_n:=U_1(R^n)
			:=\left\{g=(g_q)\in \widehat{R}^{n\times} \mid \iota_q(g_q) \equiv \left(\begin{smallmatrix}
				* & * \\
				0 & 1
			\end{smallmatrix}\right)\pmod{Np^n\z_q},\textrm{ for }q\mid Np^n\right\}.
		\end{equation}
		By construction, $U_{n+1}\subset U_n\subset\ldots\subset U_0$. Given any special order $R'$, we denote by $U_1(R')$ the corresponding compact open subgroup defined analogously to $U_n$.
	\end{definition}
	For any commutative ring $A$, we consider the left action of $M_2(A)$ on the polynomial ring $A[X,Y]$, defined as
	\begin{equation}
		\gamma\cdot P(X,Y):=P\left((X,Y)\gamma\right),
	\end{equation}
	for $P\in A[X,Y]$ and $\gamma\in M_2(A)$. We denote by $L_{m}(A)$ the submodule of $A[X,Y]$ consisting of homogeneous polynomials of degree $m$; by definition, $L_{m}(A)$ is stable under the action of $M_2(\z)$. Its dual module $V_{m}(A)$ is endowed with the right action
	\begin{equation}
		\mu|_{\gamma}(P(X,Y)):=\mu(\gamma\cdot P((X,Y)) ),
	\end{equation}
	for any $\mu\in V_{m}(A)$ and $P\in L_{m}(A)$.\\
	
	We take now $\mathcal{O}$ to be a finite flat extension of $\zp$, which we assume to contain all the $\phi(Np^n\ell^{2r})$-th roots of unity, where $\phi$ is Euler's totient function. Given an $\mathcal{O}$-algebra $A$, any $A$-valued Dirichlet character $\psi$ modulo $Np^n\ell^{2r}$, can be lifted to $\psi_{\A}:\q^\times\backslash \A_{\q}^\times\longrightarrow A^\times,$ its \emph{ad\`{e}lization}, that is, the unique finite order Hecke character
	\begin{equation}
		\psi_{\A}:\q^\times\backslash \A_{\q}^\times/\mathbb{R}_{+}(1+Np^n\ell^{2r}\widehat{\z})^{\times}\longrightarrow A^\times
	\end{equation}
	such that $\psi_{\A}(\varpi_q)=\varepsilon^{-1}(q)$, for every $q\nmid Np^n\ell^{2r}$ and $\varpi_q=(\varpi_{q,l}\mid \varpi_{q,l}=1\textrm{ if }l\neq q, \varpi_{q,q}=q)\in \A_{\q,f}^\times$. 
	
	We fix such a Dirichlet character $\psi$ modulo $Np^n\ell^{2r}$ with small conductor at $\ell$. More precisely, as in \cite{HPS1989basis}, we enforce the following assumption. 
	\begin{assumption}\label{psi hyp}
		The $\ell$-component of $\psi$, $\psi_{\ell^{2r}}$, is either the trivial character modulo $\ell$ or an odd character of conductor exactly $\ell$.
	\end{assumption}
	\begin{definition}\label{quaternionic mf k def}
		Let $k\geq 2$ be an integer and let $U_n$ be as in Definition \ref{U1(R)}. For an $\mathcal{O}$-algebra $A$ and an $A$-valued Dirichlet character $\psi=\psi_{Np^n}\psi_{\ell^{2r}}$ (for $\psi_{Np^n}$ modulo $Np^n$ and $\psi_{\ell^{2r}}$ modulo $\ell^{2r}$) we define the space of $p$-adic quaternionic modular forms of weight $k$, level $Np^n\ell^{2r}$ and character $\psi$ as
		\begin{multline}
			S_k(U_n,\psi,A):=\Big\{\varphi:\widehat{B}^\times\longrightarrow V_{k-2}(A)\mid \varphi(b\tilde{b}uz)=\psi_{Np^n,\A}^{-1}(z)\psi_{\ell^{2r},\A}^{-1}(z)z_p^{k-2}\widetilde{\psi_{\ell^{2r}}}(u_\ell)\varphi(\tilde{b})|_{u_p},\\ 
			\text{ for } b\in B^\times,\ph{.}\tilde{b}\in \widehat{B}^\times,\ph{.} u\in U_n,\ph{.}z\in \A_{\q,f}^\times \Big\}.
		\end{multline}
	\end{definition}
	\noindent This space can be identified with the space of functions $\varphi:B^\times\backslash \widehat{B}^\times\longrightarrow V_{k-2}(A)$ satisfying
	\begin{equation}
		\varphi(z\tilde{b})=\psi_{Np^n,\A}^{-1}(z)\psi_{\ell^{2r},\A}^{-1}(z)z_p^{k-2}\varphi(\tilde{b})
	\end{equation}
	for $\tilde{b}\in B^\times\backslash\widehat{B}^\times$ and $z\in \A_{\q,f}^\times$, and such that $(u\cdot\varphi)(\tilde{b}):=\varphi|_{u_p^{-1}}(\tilde{b}u)=\widetilde{\psi_{\ell^{2r}}}(u_\ell)\varphi(\tilde{b}),$ for any $u\in U_n$ and any $\tilde{b}\in \widehat{B}^\times$.
	
	\begin{remark}
	    The term \emph{$p$-adic quaternionic modular forms} refers to the fact that we are considering $p$-adic coefficients, and the action at the place $p$ instead of the action at infinity (see the next Section \ref{quaternionic mf of higher weight}).
	\end{remark}
	
	\subsection{Quaternionic modular forms of higher weight}\label{quaternionic mf of higher weight}
	
	The definition of classical quaternionic modular forms for higher weight is similar to the one for weight 2. We fix an identification $\iota_\infty:B_\infty\hookrightarrow M_2(\cc)$ in order to compare $p$-adic and classical quaternionic modular forms.
	\begin{definition}\label{quaternionic mf wt k def}
		We define the space of weight-$k$, $k\geq 2$, quaternionic modular forms with level structure $R^n(\A_\q)^\times$, character $\chi$ satisfying Assumption \ref{chi-R} and $\cc$-coefficients, as the $\cc$-vector space $S_k(R^n, \widetilde{\chi})$ of all continuous functions $\varphi_\infty:B(\A_\q)^\times\longrightarrow V_{k-2}(\cc)$ satisfying
		\begin{equation}
			\varphi_\infty(b\tilde{b}b_\infty r)=\widetilde{\chi}^{-1}(r)|n(b_\infty^{-1})|_\A^{(k-2)/2}b_\infty^{-1}\cdot\varphi_\infty(\tilde{b})
		\end{equation}
		for all $b\in B^\times$, $b_\infty\in B_\infty^\times$, $\tilde{b}\in B(\A_\q)^\times$ and $r\in R^n(\A_\q)^\times$.
	\end{definition}
	As explained in Chapter 2 of \cite{Hida1988b}, we can identify classical and $p$-adic modular forms. We identify $\cp$ with $\cc$ compatibly with the fixed inclusion $\overline{\q}_p\hookrightarrow \cc$ and associate (see also \cite{Hsieh2021}, Eq. (4.4)) to $\varphi\in S_k(U_n,\psi,\cp)$ the form $\Phi_{\infty}(\varphi)\in S_k(R^n, \widetilde{\psi})$, defined as
	\begin{equation}
		\Phi_{\infty}(\varphi)(\tilde{b}b_\infty)=|n(\tilde{b}b_\infty)|_\A^{(k-2)/2} b_\infty^{-1}\cdot\left(\tilde{b}_p\cdot \varphi(\tilde{b})\right),
	\end{equation}
	for $b_\infty\in B_\infty^\times$ and $\tilde{b}\in B(\A_\q)^\times$.
	
	\subsection{Quaternionic Eisenstein series and newforms}\label{Quaternionic Eisenstein series and newforms}
	
	We recall the notions of quaternionic Eisenstein series and quaternionic newforms as presented in \cite{HPS1989basis}, Chapters 5 and 7. For this section, we take $A$ to be a $\qp$-module. We begin with the Eisenstein series part of $S_k(U_n,\psi,A)$, that is
	\begin{equation}
		S^{Eis}_k(U_n,\psi,A):=\Big\{f\in S_k(U_n,\psi,A)\mid \exists\ph{.} g:\A_{\q,f}^\times\longrightarrow A^{k-1}
		\textrm{ s.t. }f(\tilde{b})=g(n(\tilde{b}))\Big\},
	\end{equation}
	where $n:\widehat{B}^\times\longrightarrow \A_{\q,f}^\times$ is the extension of the quaternionic norm to $\widehat{B}$. In other words, $S^{Eis}_k(U_n,\psi,A)$ is the space of quaternionic modular forms factoring through the reduced norm map. As proved by Propositions 5.2, 5.3 and the discussion after Proposition 5.4 in \emph{loc.cit.}, this space is often trivial, in fact
	\begin{equation}
		S^{Eis}_k(U_n,\psi,A)=	\begin{cases}
			\{0\} & \textrm{if $k>2$ or $\psi_{Np^n}$ is non trivial,}\\
			A^{[\z_\ell^\times : n((R^n_\ell)^\times)]} & \textrm{if $k=2$ and $\psi_{Np^n}$ is trivial.}
		\end{cases}
	\end{equation}
	In particular, $S^{Eis}_k(U_n,\psi,A)$ has at most rank 2. Defining the Petersson inner product as in \cite{Shimizu1965} or \cite{Gross87}, one can consider the orthogonal complement of $S^{Eis}_k(U_n,\psi,A)$ in $S_k(U_n,\psi,A)$, namely
	\begin{equation}
		\mathscr{S}_k(U_n,\psi,A)
		:= \begin{cases}                  
			S_k(U_n,\psi,A)/S^{Eis}_k(U_n,\psi,A) & \textrm{if $k=2$ and $\psi_{Np^n}$ is trivial,}\\
			S_k(U_n,\psi,A) & \textrm{otherwise.}
		\end{cases}
	\end{equation}
	\begin{definition}
		We call $\mathscr{S}_k(U_n,\psi,A)$ the space of $A$-valued cuspidal quaternionic modular forms of level $U_n$ and character $\psi$.
	\end{definition}
	Inside of this space, we find the so-called space of old forms, $\mathscr{S}^{old}_k(U_n,\psi,A)$, defined to be the subspace of $\mathscr{S}_k(U_n,\psi,A)$ spanned by all $\mathscr{S}_k(U_1(R'),\psi,A)$ for each special order $R'\subset R_n$ for which $S_k(U_1(R'),\psi,A)$ makes sense. One should pay attention to fix a suitable ramified extension of $\q_\ell$, but we point the reader to Remarks 7.13 and 7.14 of \cite{HPS1989basis} for further details. Finally, we define the space of quaternionic newforms $\mathscr{S}^{new}_k(U_n,\psi,A)$ as the orthogonal complement of $\mathscr{S}^{old}_k(U_n,\psi,A)$, with respect to the Petersson inner product, inside $\mathscr{S}_k(U_n,\psi,A)$.
	
	\section{Hecke algebras and lifts to quaternionic modular forms}
	
	 One of Hida's main results is the extension of the classical duality between the Hecke algebra and the space of classical modular forms to $p$-adic families. The analogous result can be recovered in the quaternionic setting, when one considers Eichler orders or special orders with odd exponent at the primes of ramification, but in the case of special orders with even exponent, this is no more true (see Remark \ref{no-duality and Hecke ext}). Even though one cannot speak about duality anymore, it is indeed possible to recover the correct dimension result for proving a rank-2 Hida theory.
	
	\subsection{Hecke operators}\label{Hecke operators}
	
	For any prime $q$, recall the element introduced in Section \ref{$p$-adic quaternionic modular forms for special orders}, $\varpi_q\in\A_{\q,f}^{\times}$ such that $\varpi_{q,q}=q$ and $1$ otherwise. Let $A$ be again an $\mathcal{O}$-algebra and take $\varphi\in S_k(U_n,A)$. On this quaternionic space we have (for any $\tilde{b}\in\widehat{B}^\times$) the Hecke operators $T_q$,
	\begin{equation}
		T_q \varphi(\tilde{b}) =\begin{cases}
			\varphi\left(\tilde{b}\mat{1}{0}{0}{\varpi_q}\right) + \displaystyle\sum_{a\in\z/q\z}\varphi\left(\tilde{b}\mat{\varpi_q}{a}{0}{1}\right) & \textrm{ for each prime }q\nmid Np^n\ell^{2r},\\
			\\
			\varphi|_{\mat{1}{0}{0}{p}}\left(\tilde{b}\mat{1}{0}{0}{\varpi_p}\right) + \displaystyle\sum_{a\in\z/q\z}\varphi|_{\mat{p}{a}{0}{1}}\left(\tilde{b}\mat{\varpi_p}{a}{0}{1}\right) & \textrm{ for }q=p \textrm{ and }n=0,
		\end{cases}
	\end{equation}
	and the Hecke operators $U_q$,
	\begin{equation}
		U_q \varphi(\tilde{b}) =\begin{cases}
			\displaystyle\sum_{a\in\z/q\z}\varphi\left(\tilde{b}\mat{\varpi_q}{a}{0}{1}\right) & \text{ for }q\mid N,\\
			\\
			\displaystyle\sum_{a\in\z/p\z}\varphi|_{\mat{p}{a}{0}{1}}\left(\tilde{b}\mat{\varpi_p}{a}{0}{1}\right) & \textrm{ for }q=p \textrm{ and }n>0.
		\end{cases}
	\end{equation}
	We also consider the quaternionic operator at $\ell$, $\tilde{U}_\ell$ which is defined as
	\begin{equation}
		\tilde{U}_\ell \varphi(\tilde{b}) =\varphi\left(\tilde{b}\tilde{\varpi}_\ell\right),
	\end{equation}
	for $\tilde{\varpi}_\ell$ such that $\tilde{\varpi}_{\ell,\ell}$ is a units in the maximal order at $\ell$ of norm $n(\tilde{\varpi}_{\ell,\ell})=\ell$, and $\tilde{\varpi}_{\ell,q}=1$ elsewhere. For each $d\in \Delta_{N p^n\ell^{2r}}:= (\z/Np^n\ell^{2r}\z)^\times$, we also recall the diamond operator $\langle d\rangle$ with its usual definition on classical modular forms and straightforwardly extended to the $p$-adic quaternionic case.
	On the space of classical modular forms $S_k(\Gamma_1(Np^n\ell^{2r}),\psi,A)$ we have the usual operators with a similar expression to the quaternionic ones except at $\ell$, where the definition of $U_\ell$ is analogous to the above $U_q$ operators.
	
	\subsection{Hecke algebras}\label{Hecke algebras}
	
	For each $n\geq 1$, let $\mathsf{H}^1_n(A)$ be the Hecke algebra generated over $A$ by all Hecke operators and the diamond operators away from the level, which acts on $S_k(\Gamma_1(Np^n\ell^{2r}),A)$. We denote by $\mathsf{H}_n(A)$ the direct summand of $\mathsf{H}^1_n(A)$ acting on $S_k(\Gamma_1(Np^n\ell^{2r}),\psi,A)$ and by $\mathsf{h}_n(A)$ the Hecke algebra acting on the space of newforms $S_k^{new}(\Gamma_1(Np^n\ell^{2r}),\psi,A)$. For each $m>n$ we have the projection maps $\mathsf{H}^{1}_m(A)\twoheadrightarrow \mathsf{H}^1_n(A)$ and the same holds true for the subalgebras $\mathsf{H}_n(A)$ and $\mathsf{h}_m(A)$. We construct the projective limits with respect to these maps,
	\begin{align}
		& \mathsf{H}^1_\infty(A)=\varprojlim \mathsf{H}^1_n(A), & \mathsf{H}_\infty(A)=\varprojlim \mathsf{H}_n(A) && \text{and} & &\mathsf{h}_\infty(A)=\varprojlim \mathsf{h}_n(A),
	\end{align}
	together with the projection maps $\mathsf{H}^1_\infty(A)\twoheadrightarrow \mathsf{H}_\infty(A)\twoheadrightarrow \mathsf{h}_\infty(A)$. For any $n\geq 1$, we define $\mathsf{H}^{1,ord}_n(A)$ to be ordinary part of $\mathsf{H}^1_n(A)$, namely the product of all the localizations of $\mathsf{H}^1_n(A)$ on which $U_p$ is invertible, and denote by $\mathsf{e}_n$ the corresponding projector $\mathsf{e}_n:\mathsf{H}^{1}_n(A)\twoheadrightarrow \mathsf{H}^{1,ord}_n(A)$. Similarly we define $\mathsf{H}^{ord}_n(A)$ and $\mathsf{h}^{ord}_n(A)$, together with the corresponding ordinary projectors, which we denote by the same symbol $\mathsf{e}_n$. Passing to the limit we obtain $\mathsf{H}^{1,ord}_\infty(A)$, $\mathsf{H}^{ord}_\infty(A)$ and $\mathsf{h}^{ord}_\infty(A)$, each of them equipped with the corresponding ordinary projector $\mathsf{e}_\infty=\varprojlim \mathsf{e}_n$.
	\begin{remark}\label{remark ul=0}
		It is well known (see \emph{e.g.} \cite{Li1975}, Theorem 3) that Assumption \ref{psi hyp} forces the Hecke operator $U_\ell$ to be trivial on the space of classical modular forms of level $\ell^{2r}$. We can then identify each $\mathsf{h}_n(A)$ with the Hecke $A$-algebra
		\begin{equation}
			\mathsf{h}^{(\ell)}_n(A)\subseteq End(S_k(\Gamma_1(Np^n\ell^{2r}),\psi,A))
		\end{equation}
		generated by all diamond and Hecke operators, except $U_\ell$.
	\end{remark}
	
	On the quaternionic side we proceed similarly. For each $n\geq 1$, let $\mathsf{H}^{B}_n(A)$ be the Hecke algebra acting on $S_k(U_n,\psi,A)$, generated over $A$ by the Hecke and diamond operators. We denote by $\mathsf{h}_n^{B}(A)$ the component acting on the space of newforms $\mathscr{S}_k^{new}(U_n,\psi,A)$. For each $n\geq 1$ we have the projection maps $\mathsf{H}^{B}_n(A)\twoheadrightarrow \mathsf{h}_n^{B}(A)$ and we construct the projective limits with respect to these maps,
	\begin{align}
		& \mathsf{H}^{B}_\infty(A)=\varprojlim \mathsf{H}^{B}_n(A) && \text{ and } & &\mathsf{h}^{B}_\infty(A)=\varprojlim \mathsf{h}^{B}_n(A),
	\end{align}
	together with the projection map $\mathsf{H}^{B}_\infty(A)\twoheadrightarrow \mathsf{h}^{B}_\infty(A)$. In the end, we define as above the ordinary Hecke algebras $\mathsf{H}^{B,ord}_n(A)$ and $\mathsf{h}^{B,ord}_n(A)$, and obtain $\mathsf{H}^{B,ord}_\infty(A)$ and $\mathsf{h}^{B,ord}_\infty(A)$ as inverse limit of the $\mathsf{H}^{B,ord}_n(A)$ and $\mathsf{h}_n^{B,ord}(A)$ respectively.
	
	The Jacquet--Langlands correspondence provides a compatible morphism between the classical and the quaternionic side, that is
	\begin{equation}
		JL_{\infty}:\mathsf{h}_\infty(A)\longrightarrow \mathsf{H}^{B}_\infty(A),
	\end{equation}
	which preserves the Hecke and diamond operators away from the discriminant of the quaternion algebra.
	
	Let $\tilde{\Lambda}=\mathcal{O} [\![ \zp^\times ]\!]$ be the finite flat extension of the Iwasawa algebra $\zp[\![ \zp^\times ]\!]$ obtained from $\mathcal{O}$. We remark that by construction, the algebra $\mathsf{H}^{ord}_\infty$ is naturally a $\tilde{\Lambda}$-algebra; moreover, one can prove that it is finitely generated over $\tilde{\Lambda}$.
	We define the two universal $\tilde{\Lambda}$-adic Hecke algebras
	\begin{align}
		\mathsf{H}_{univ}:=&\tilde{\Lambda}[T_q, U_l, \langle d\rangle, \text{ for }q\nmid Np^n, l\mid Np\ell^{2r}, d\in \Delta_{Np^n\ell^{2r}}],\\ \mathsf{H}_{univ}^B:=&\tilde{\Lambda}[\tilde{U}_\ell, T_q, U_l, \langle d\rangle, \text{ for }q\nmid Np^n, l\mid Np, d\in \Delta_{Np^n\ell^{2r}}]
	\end{align}
	and, as in \cite{LongoVigni2012}, we obtain the compatible morphisms 
	\begin{align}
		& \mathsf{H}_{univ}\longrightarrow \mathsf{H}^{ord}_\infty(A) && \text{ and } & & \mathsf{H}_{univ}^B\longrightarrow \mathsf{H}^{B,ord}_\infty(A).
	\end{align}
	\begin{definition}
	    It is useful to formally introduce the notation considered in Remark \ref{remark ul=0}. More precisely, let $\mathsf{H}$ be any one of the Hecke algebras defined above and let $m$ be any positive integer. We denote by $\mathsf{H}^{(m)}$ the Hecke-subalgebra of $\mathsf{H}$ generated by the Hecke and diamond operators away from $M$.
	\end{definition}
	
	\subsection{Quaternionic lifts of modular forms and the failure of the duality}\label{quaternionic lifts of modular forms}
	
	We analyze more carefully \cite{HPS1989basis},  recalling the results which we need. Let $\left(\frac{-}{\ell}\right)$ be the Kronecker character at $\ell$ and $F=Frac(\mathcal{O})$, a field extension of $\qp$. For any space of modular forms $S_k(M,\varepsilon,F)$ and each Dirichlet character modulo $M$, we denote by $S_k(M,\varepsilon,F)^{\otimes\chi}$ the space of all the modular forms which are twists by $\chi$ of modular forms in $S_k(M,\varepsilon,F)$.
	\begin{theorem}[\cite{HPS1989basis}, Theorem 7.10]\label{JL-expl odd exp}
		Let $R'$ be a special order of level $M\ell^{2r+1}$ (so $L_\ell$ is the unramified quadratic extension of $\q_\ell$). Let $\varepsilon$ be a character modulo $N$ such that $\varepsilon_{\ell^{2r+1}}$ is either the trivial character modulo $\ell$ or an odd character modulo $\ell$. Suppose moreover that $\varepsilon$ is even and that $r\geq \nu_\ell(cond(\varepsilon_\ell))$. Then there exist a Hecke--equivariant isomorphism
		\begin{equation}
			\mathscr{S}_k^{new}(U_1(R'),\widetilde{\varepsilon},\mathbb{C})\cong S_k^{new}(\Gamma_1(M \ell^{2r+1}),\varepsilon,\mathbb{C}).
		\end{equation}
	\end{theorem}
	The above theorem proves that, as in the case of Eichler orders, there is a one-to-one correspondence for special orders with odd exponent at $\ell$. The situation for even exponent is more complicated.
	\begin{theorem}[\cite{HPS1989basis}, Theorems 7.16 \& 7.17]\label{Thm 7.16-7}
		Let $\ell$ be an odd prime, and let $r\geq 1$ and $k\geq 2$ be integers. Let $\psi$ be a character modulo $Np^n\ell^{2r}$ such that $\psi(-1)=(-1)^k$, it satisfies Assumptions \ref{chi-R} and \ref{psi hyp}, and such that $cond_\ell(\psi)\leq 2r-1$. Then the following decomposition of $\mathsf{h}_n^{(Np^n\ell^{2r})}(\mathbb{C})$-modules holds true.
		\begin{enumerate}[label=(\alph*)]
			\item \label{Thm 7.16-7.a} If $r=1$ and $\psi_{\ell^{2}}$ is the trivial character:
			\begin{multline}
				2 S_k^{new}(\Gamma_1(Np^n\ell^{2}),\psi,\mathbb{C})
				\cong \mathscr{S}^{new}_k(U_n,\psi,\mathbb{C})\oplus S_k^{new}(\Gamma_1(Np^n\ell),\psi,\mathbb{C})^{\otimes\left(\frac{-}{\ell}\right)}\oplus\\
				2 S_k^{new}(\Gamma_1(Np^n),\psi,\mathbb{C})^{\otimes\left(\frac{-}{\ell}\right)} \oplus \bigoplus_{\chi /\sim}2 S_k^{new}(\Gamma_1(Np^n\ell),\chi^2\psi,\mathbb{C})^{\otimes\overline{\chi}}
			\end{multline}
			where the sum $\bigoplus_{\chi /\sim}$ runs over all the $\frac{1}{2}(\ell-3)$ classes of primitive characters modulo $\ell$ excepting $\left(\frac{-}{\ell}\right)$, modulo the equivalence $\chi\sim \overline{\chi}$.
			
			\item \label{Thm 7.16-7.b} If $r=1$ and $\psi_{\ell^{2}}$ is a odd character modulo $\ell$:
			\begin{equation}
				2 S_k^{new}(\Gamma_1(Np^n\ell^{2}),\psi,\mathbb{C}) \cong \mathscr{S}^{new}_k(U_n,\widetilde{\psi},\mathbb{C})\oplus \bigoplus_{\chi /\sim}2 S_k^{new}(\Gamma_1(Np^n\ell),\chi^2\psi,\mathbb{C})^{\otimes\overline{\chi}}
			\end{equation}
			where $\widetilde{\psi}$ is a lift of $\psi$ as in Section \ref{lift character} and the sum $\bigoplus_{\chi /\sim}$ runs over all the $\frac{1}{2}(\ell-3)$ classes of primitive characters modulo $\ell$ excepting $\overline{\psi_{\ell^{2}}}$, modulo the equivalence $\chi\sim \overline{\chi\psi_{\ell^{2}}}$.
			
			\item \label{Thm 7.16-7.c} If $r\geq 2$ and $\psi_{\ell^{2r}}$ is either trivial or odd of conductor $\ell$:
			\begin{equation}
				2 S_k^{new}(\Gamma_1(Np^n\ell^{2r}),\psi,\mathbb{C}) \cong \mathscr{S}^{new}_k(U_n,\widetilde{\psi},\mathbb{C})\oplus \bigoplus_{\chi}2 S_k^{new}(\Gamma_1(Np^n\ell^{r}),\chi^2\psi,\mathbb{C})^{\otimes\overline{\chi}}
			\end{equation}
			where $\widetilde{\psi}$ is a lift of $\psi$ as in Section \ref{lift character} and the sum $\bigoplus_{\chi}$ runs over all the $\ell^r-2\ell^{r-1}+\ell^{r-2}$ classes of primitive characters modulo $\ell^r$, modulo the equivalence $\chi\sim \overline{\chi\psi_{\ell^{2r}}}$.
		\end{enumerate}
	\end{theorem}
	\begin{remark}\label{no-duality and Hecke ext}
		\begin{enumerate}[label=(\alph*)]
			\item \label{hecke ext} In the above theorem the decomposition is given as $\mathsf{h}_n^{(Np^n\ell^{2r})}(\mathbb{C})$-modules, but strong multiplicity one for classical modular newforms guarantees the decomposition to hold (at least) as $\mathsf{h}_n^{(\ell)}(\mathbb{C})$-modules. As already noticed in Remark \ref{remark ul=0}, the Hecke algebra $\mathsf{h}_n^{(\ell)}(\mathbb{C})$ coincides with $\mathsf{h}_n(\mathbb{C})$ since the Hecke operator $U_\ell$ is the $0$-operator on this space.
			\item \label{no duality} The theorem implies that the duality between the Hecke algebra and the space of modular forms does not necessarily hold true for special orders with level $\ell^{2r}$. This situation represents the main difference between this setting and the case of classical modular forms (and special orders with an odd power of $\ell$). We recall that, on the contrary, the Jacquet--Langlands correspondence does hold true, as well as the multiplicity one result for automorphic representations. This phenomenon is purely local, as already remarked in Example 2.6 of \cite{LRdvP2018}. More precisely, the dimension of the local automorphic representation at $\ell$ is bigger than $1$ in the case of level $\ell^{2r}$ and determined by the minimal conductor of the modular forms. We refer to Section 5 of \cite{Carayol1984} for all the related details.
		\end{enumerate}
	\end{remark}
	We recall the definition of a twist-minimal modular form.
	\begin{definition}\label{twist min}
		A modular form is twist-miminal if $cond(\pi_{f,q})\geq cond(\pi_{f,q}\otimes \chi)$ for all $q$-adic characters $\chi$, where $\pi_f = \otimes'_q\pi_{f,q}$ is the automorphic representation attached to $f$ and for any automophic representation $\pi= \otimes'_q\pi_{q}$ we denote $cond(\pi_q)$ the conductor of $\pi_q$.
	\end{definition}
	\begin{corollary}\label{mult 2}
		Each twist-minimal modular eigenform in $S_k^{new}(\Gamma_1(Np^n\ell^{2r}),\psi,\mathbb{C})$ lifts to (up to linear combinations) exactly two linearly independent quaternionic modular eigenforms in $\mathscr{S}^{new}_k(U_n,\widetilde{\psi},\mathbb{C})$ with the same Hecke eigenvalues for $\mathsf{h}_n^{(\ell)}(\cc)$.
	\end{corollary}
	
	Regardless of Remark \ref{no-duality and Hecke ext}.\ref{no duality}, one can still obtain an isomorphism between the space of quaternionic modular forms and the square of a suitable Hecke algebra, as in the following proposition.
	\begin{proposition}\label{square iso hecke-mf}
		Under the hypotheses of Theorem \ref{Thm 7.16-7}, there exists a $\mathbb{C}$-vector subspace $T_k(n,r,\psi)$ of $S_k^{new}(\Gamma_1(Np^n\ell^{2r}),\psi,\mathbb{C})$, which is a $\mathsf{h}_n^{(\ell)}(\cc)$-submodule satisfying
		\begin{equation}
			2T_k(n,r,\psi)\cong 
			\begin{cases}
				\mathscr{S}^{new}_k(U_n,\psi,\mathbb{C})\oplus S_k^{new}(\Gamma_1(Np^n\ell),\psi,\mathbb{C})^{\otimes\left(\frac{-}{\ell}\right)} & \textrm{ if $r=1$ and $\psi_\ell$ is trivial,}\\
				\mathscr{S}^{new}_k(U_n,\psi,\mathbb{C}) & \textrm{ otherwise.}
			\end{cases}
		\end{equation}
		Moreover, for $\mathsf{h}_n^{T}(\mathbb{C})$ the Hecke-subalgebra of $\mathsf{h}_n(\mathbb{C})$ acting on $T_k(n,r,\psi)$, we have an isomorphism of $\mathsf{h}_n^{T}(\mathbb{C})=\mathsf{h}_n^{T,(\ell)}(\mathbb{C})$-modules,
		\begin{equation}
			\left(\mathsf{h}_n^{T}(\mathbb{C})\right)^2\cong
			\begin{cases}
				\mathscr{S}^{new}_k(U_n,\psi,\mathbb{C})\oplus S_k^{new}(\Gamma_1(Np^n\ell),\psi,\mathbb{C})^{\otimes\left(\frac{-}{\ell}\right)} & \textrm{ if $r=1$ and $\psi_\ell$ is trivial,}\\
				\mathscr{S}^{new}_k(U_n,\psi,\mathbb{C}) & \textrm{ otherwise.}
			\end{cases}
		\end{equation}
	\end{proposition}
	\begin{proof}
		The first statement follows directly from Theorem \ref{Thm 7.16-7} as noticed in Chapter 8 of \cite{HPS1989basis}. The second part follows from $	Hom\left(2T_k(n,r,\psi),\mathbb{C}\right)\cong Hom\left(T_k(n,r,\psi),\mathbb{C}\right)^2\cong 2T_k(n,r,\psi),$
		where the first isomorphism is due to the properties of $Hom(-,\mathbb{C})$ and the second is the Hecke-duality for classical modular forms restricted to $T_k(n,r,\psi)$ (since the decomposition is Hecke-equivariant away from $\ell$).
	\end{proof}
	As in Section \ref{Hecke algebras}, taken $A$ an $\mathcal{O}$-algebra, we define $\mathsf{h}^T_{\infty}(A)=\varprojlim \mathsf{h}_{n}^T(A)$ and $\mathsf{h}^{T,ord}_{\infty}(A)=\varprojlim \mathsf{h}_{n}^{T,ord}(A)$. We obtain injective homomorphisms $\mathsf{h}^T_{\infty}(A)\hookrightarrow \mathsf{h}_{\infty}(A)$ and $\mathsf{h}^{T,ord}_{\infty}(A)\hookrightarrow \mathsf{h}^{ord}_{\infty}(A)$.
	\begin{definition}
		For any module $M$ with an action of a suitable Hecke algebra, and any classical eigenform $g$, we denote by $M[g]$ the $g$-isotypic component of $M$, \emph{i.e.} the biggest submodule of $M$ on which the Hecke algebra acts with the same eigenvalues of $g$.
	\end{definition}
	\begin{proposition}\label{pointwise dim}
		Let $g$ be a newform in $S_k^{new}(\Gamma_1(Np^n\ell^{2r}),\psi,F)$ with $k\geq 2$ and $\psi(-1)=(-1)^{k}$. Write $\psi=\psi_{Np^n}\psi_{\ell^{2r}}$ for $\psi_{Np^n}$ and $\psi_{\ell^{2r}}$ the component of $\psi$, respectively, modulo $Np^n$ and $\ell^{2r}$.
		\begin{enumerate}[label=(\alph*)]
			\item If $r=1$ and $\psi_{\ell^{2}}$ is the trivial character modulo $\ell$,
			\begin{multline}
				\dim_{F}\left(\mathscr{S}^{new}_k(U_n,\widetilde{\psi},F)[g]\right)=\\
				\dim_{F}\left(\mathscr{S}_k(U_n,\widetilde{\psi},F)[g]\right)=
				\begin{cases}
					2 & \textrm{ if $g$ is twist-minimal at $\ell$},\\
					1 & \textrm{ if } g\in S_k^{new}(\Gamma_1(Np^n\ell),\psi,F)^{\otimes\left(\frac{-}{\ell}\right)},\\
					0 & \textrm{ otherwise.}
				\end{cases}
			\end{multline}
			\item If either $\psi_{\ell^{2r}}$ is a non-trivial character modulo $\ell$ or $r\geq 2$,
			\begin{equation}
				\dim_{F}\left(\mathscr{S}^{new}_k(U_n,\widetilde{\psi},F)[g]\right)
				=\dim_{F}\left(\mathscr{S}_k(U_n,\widetilde{\psi},F)[g]\right)=  \begin{cases}
					2 & \textrm{ if $g$ is twist-minimal at $\ell$},\\
					0 & \textrm{ otherwise.}
				\end{cases}
			\end{equation}
		\end{enumerate}
	\end{proposition}
	\begin{proof}
		This is a straightforward consequence of Theorem \ref{Thm 7.16-7} combined with the fact that strong multiplicity one applies to $g$.
	\end{proof}
	
	\subsection{Choice of a modular form}\label{choice of modular form}
	Let $f\in S_2(\Gamma_1(Np^n\ell^{2r}),\psi,\mathbb{C})$ be a fixed $p$-ordinary newform, for $n\geq 1$ and with $\psi$ a Dirichlet character modulo $Np^n\ell^{2r}$ satisfying Assumptions \ref{chi-R} and \ref{psi hyp}. In this way, the automorphic representation associated with $f$ admits a Jacquet--Langlands lift. Moreover, we assume that the $p$-adic Galois representation associated with $f$ is residually absolutely irreducible and $p$-distinguished. Let $F=\qp(f)$ be the finite extension of $\qp$ defined by $f$ and take $\mathcal{O}$ to be its ring of integers; note that $\mathcal{O}$ is a finite flat extension of $\zp$. We denote by $f_\infty$ the unique Hida family passing through $f$. By duality with the ordinary Hecke algebra, we know that $f_\infty$ defines a character, which we denote with the same symbol $f_\infty$,
	\begin{equation}
		f_\infty: \mathsf{H}^{ord}_\infty(F)\longrightarrow \mathcal{R},
	\end{equation}
	where $\mathcal{R}$ is the \emph{universal ordinary $p$-adic Hecke algebra of tame level $N\ell^{2r}$} as in Definition 2.4 of \cite{GreenbergStevens1993}. The Jacquet--Langlands correspondence ensures that such character factors through the morphism to $\mathsf{H}^{B}_\infty(F)$; we keep denoting the corresponding map by $f_\infty: \mathsf{H}^{B,ord}_\infty(F)\longrightarrow \mathcal{R}$.
	
	\section{The control theorem}
	
	In this last section, we prove a \emph{control theorem} for special orders of even conductor at $\ell$. We introduce a space that is suitable for the $p$-adic interpolation and we define some specialization maps. We consider again $\mathcal{O}$ to be the ring of integers of a fixed finite extension of $\qp$ and we take an $\mathcal{O}$-algebra which we denote again by $A$.
	
	\subsection{Specialization maps}\label{specialization maps}
	Let $\mathsf{X}=(\zp\times \zp)^{prim}$ be the set of primitive row vectors, that is, the vectors in $\zp^2$ which have at least one component not divisible by $p$. Denote by $\mathscr{C}(\mathsf{X},A)$ the space of $A$-valued continuous functions on $\mathsf{X}$ and by $\mathcal{M}(\mathsf{X},A)$ the space of $A$-valued measures on $\mathsf{X}$. We have a left $M_2(\zp)$-action on $\mathscr{C}(\mathsf{X},A)$ via
	\begin{equation}
		\gamma\cdot f(x,y)=f\left((x,y)\gamma\right),
	\end{equation}
	for $f\in\mathscr{C}(\mathsf{X},A)$ and $\gamma\in M_2(\zp)$, and the induced right action on $\mathcal{M}(\mathsf{X},A)$ as
	\begin{equation}
		\mu|_{\gamma}(f(x,y))=\mu(\gamma\cdot f(x,y)),
	\end{equation}
	for $\mu \in \mathcal{M}(\mathsf{X},A)$. Considering the action by $U_n$, with $n\geq 1$, we can notice that the subspace $p\zp\times\zp^\times \subset \mathsf{X}$ satisfies 
	\begin{equation}
		(p\zp\times\zp^\times) \cdot (U_n)_p= (p\zp\times\zp^\times) \mat{\zp^\times}{\zp}{p^n\zp}{1+p^n\zp}=p\zp\times\zp^\times.
	\end{equation}
	\begin{definition}\label{measure valued quat forms}
		Let $\psi$ be a Dirichlet character modulo $N\ell^{2r}$ satisfying Assumptions \ref{chi-R} and \ref{psi hyp}. We define the measure-valued quaternionic modular forms with character $\psi$ as the space
		\begin{multline}
			S_2(U_0,\psi, \mathcal{M}(\mathsf{X},A)):=\Big\{\varphi:\widehat{B}^\times\longrightarrow \mathcal{M}(\mathsf{X},A)\mid
			\varphi(b\tilde{b}uz)=\psi_{N,\A}^{-1}(z)\psi_{\ell^{2r},\A}^{-1}(z)\widetilde{\psi_{\ell^{2r}}}(u_\ell)\varphi(\tilde{b})|_{u_p},\\
			\text{ for } b\in B^\times,\ph{.}\tilde{b}\in \widehat{B}^\times,\ph{.} u\in U_0,\ph{.}z\in \A_{\q,f}^\times \Big\}.
		\end{multline}
	\end{definition}
	\noindent This space can be identified with the space of functions $\varphi:B^\times\backslash \widehat{B}^\times\longrightarrow \mathcal{M}(\mathsf{X},A)$ satisfying
	\begin{equation}
		\varphi(z\tilde{b})=\psi_{N,\A}^{-1}(z)\psi_{\ell^{2r},\A}^{-1}(z)\varphi(\tilde{b}),
	\end{equation}
	for $\tilde{b}\in B^\times\backslash\widehat{B}^\times$ and $z\in \A_{\q,f}^\times$, and such that $\varphi|_{u_p^{-1}}(\tilde{b}u)=\widetilde{\psi_{\ell^{2r}}}(u_\ell)\varphi(\tilde{b})$ for any $u\in U_0$ and any $\tilde{b}\in B^\times\backslash\widehat{B}^\times$.\\
	
	Take $k\geq 2$ and let $\varepsilon:\zp^\times \longrightarrow A^\times$ be any character which factors through $(\zp/p^m\zp)^\times$. We extend $\varepsilon$ multiplicatively to $\zp$ imposing $\varepsilon(p)=0$. We define the specialization map
	\begin{equation}
		\nu_{k,\varepsilon}: S_2(U_0,\psi, \mathcal{M}(\mathsf{X},A))\longrightarrow S_k(U_n,\psi\varepsilon, A(\varepsilon))
	\end{equation}
	such that
	\begin{equation}
		\nu_{k,\varepsilon}(\varphi)(\tilde{b})(P):=\int_{p\zp\times \zp^\times}\varepsilon(y)P(x,y)d(\varphi(\tilde{b}))(x,y),
	\end{equation}
	where $n=\max\{1,m\}$, $\varphi\in S_2(U_0,\psi, \mathcal{M}(\mathsf{X},A))$, $\tilde{b}\in B^\times\backslash\widehat{B}^\times$ and $P\in L_{k-2}(A)$.
	\begin{proposition}\label{specialization maps prop}
		The specialization maps $\nu_{k,\varepsilon}$ are well-defined and Hecke-equivariant for $ \mathsf{H}^{B}_{univ}$, where the equivariance at $p$ is meant as $\nu_{k,\varepsilon}(T_p\varphi)=U_p\nu_{k,\varepsilon}(\varphi)$.
	\end{proposition}
	\begin{proof}
		Let $\varphi\in S_2(U_0,\psi, \mathcal{M}(\mathsf{X},A))$. Then, for any $\tilde{b}\in B^\times\backslash\widehat{B}^\times$, $z\in \A_{\q,f}^\times$, $u\in U_n$ and $P\in L_{k-2}(A)$, we have
		\begin{multline}
			\nu_{k,\varepsilon}(\varphi)(\tilde{b}uz)|_{u_p^{-1}}(P) =\int_{p\zp\times\zp^\times}\varepsilon(y)(u_p^{-1}\cdot P(x,y))d(\varphi(\tilde{b}uz))(x,y) \\
			=\psi_{N,\A}^{-1}(z)\psi_{\ell^{2r},\A}^{-1}(z)\widetilde{\psi_{\ell^{2r}}}(u_\ell)
			\cdot\int_{p\zp\times\zp^\times}(u_pz_p)\cdot \left(\varepsilon(y)P((x,y)u_p^{-1})\right)d(\varphi(\tilde{b}))(x,y)
		\end{multline}
		and, since $(x,y)u_p^{-1}=(*,y+p*)$, $\varepsilon$ is extended to $\zp$ and $P((x,y)z_p)=z_p^{k-2}P(x,y)$, we obtain
		\begin{multline}
			\psi_{N,\A}^{-1}(z)\psi_{\ell^{2r},\A}^{-1}(z)\widetilde{\psi_{\ell^{2r}}}(u_\ell)\int_{p\zp\times\zp^\times}\varepsilon(y)\varepsilon(z_p)\left(P((x,y)u_p^{-1}u_pz_p)\right)d(\varphi(\tilde{b}))(x,y)\\
			=\psi_{N,\A}^{-1}(z)\varepsilon_{\A}(z)^{-1}\psi_{\ell^{2r},\A}^{-1}(z)\widetilde{\psi_{\ell^{2r}}}(u_\ell)z_p^{k-2}\int_{p\zp\times\zp^\times}\varepsilon(y)\left(P((x,y))\right)d(\varphi(\tilde{b}))(x,y)\\
			=\psi_{N,\A}^{-1}(z)\varepsilon_{\A}(z)^{-1}\psi_{\ell^{2r},\A}^{-1}(z)\widetilde{\psi_{\ell^{2r}}}(u_\ell)z_p^{k-2} \nu_{k,\varepsilon}(\varphi)(\tilde{b})(P).
		\end{multline}
		The equivariance with respect to the $T_q$ operators is obvious, as well as that for the operators $U_q$ with $q\neq p$ (also for $\tilde{U}_\ell$). To prove the equivariance at $p$ it is enough to note that we have
		\begin{multline}
			\nu_{k,\varepsilon}\varphi|_{\mat{1}{0}{0}{p}}\left(\tilde{b}\mat{1}{0}{0}{p}\right)(P) =\int_{\mathsf{X}}\chi_{p\zp\times\zp^\times}(x,y)\varepsilon(y)P\left(x,y\right)d\left(\varphi|_{\mat{1}{0}{0}{p}}\left(\tilde{b}\mat{1}{0}{0}{p}\right)\right)(x,y) \\
			=\int_{\mathsf{X}}\chi_{p\zp\times\zp^\times}(x,py)\varepsilon(py)P\left(x,py\right)d\left(\varphi\left(\tilde{b}\mat{1}{0}{0}{p}\right)\right)(x,y)=0
		\end{multline}
		for $\chi_{p\zp\times\zp^\times}(x,y)$ the characteristic function of $p\zp\times\zp^\times$.
	\end{proof}
	
	We must now investigate the properties of the space $S_2(U_0,\psi, \mathcal{M}(\mathsf{X},A))$. We begin by noticing that the action on $\mathscr{C}(\mathsf{X},A)$ is exactly the one induced by the right action, 
	defined by right multiplication, of $M_2(\zp)$ on $\mathsf{X}$. We proceed similarly to Proposition 7.5 of \cite{LongoRotgerVigni2012} or Chapter 6 of \cite{GreenbergStevens1993} and, denoting by $\mathsf{X}_n$ the set of primitive vectors in $(\z/p^n\z)^2$, we recover $\mathsf{X}=\varprojlim \mathsf{X}_n$, with respect to the canonical projection maps. We obtain then $\mathcal{M}(\mathsf{X},A)=\varprojlim \mathcal{M}(\mathsf{X}_n,A)$ (see \emph{e.g.} Section 7 of \cite{MSwD1974}). Since $\mathsf{X}_n$ is a finite set, $\mathcal{M}(\mathsf{X}_n,A)$ is identified with the space $Hom(\mathsf{X}_n,A)$ of step functions. The action of $U_0$ on $\mathsf{X}_n$ is transitive and the stabilizer of $(0,1)$ is
	\begin{equation}
		Stab_{U_0}((0,1))=\{\gamma\in U_0 \mid (0,1)\gamma=(0,1)\mat{a}{b}{c}{d}=(c,d)=(0,1)\}=U_n.
	\end{equation}
	This shows that $\mathsf{X}_n=U_0/U_n=(U_0)_p/(U_n)_p$ and then that $\mathcal{M}(\mathsf{X}_n,A)=Hom_{U_n}(\mathcal{O}[U_0],A)$. We denote by $\widehat{B}^{p,\times}=B^\times\backslash \widehat{B}^\times/U_0^p$ the profinite double quotient associated with $U_0^p=U_0\cap B(\A_{\q,f}^{(p)})^\times$, for $\A_{\q,f}^{(p)}$ the ring of finite ad\`{e}les away from $p$. By Shapiro's Lemma we obtain
	\begin{equation}\label{shapiro}
		S_2(U_0,\boldsymbol{1},\mathcal{M}(\mathsf{X}_n,A))\cong\left(Hom_{\mathcal{O}}(\mathcal{O}[\widehat{B}^{p,\times}],\mathcal{M}(\mathsf{X}_n,A))\right)^{(U_0)_p}\cong S_2(U_n,\boldsymbol{1},A),
	\end{equation}
	as well as the analogous isomorphism when we consider a character $\psi$. Equation (\ref{shapiro}) implies that
	\begin{equation}\label{inverse limit measures}
		S_2(U_0,\psi,\mathcal{M}(\mathsf{X},A))=\varprojlim S_2(U_n,\psi,A),
	\end{equation}
	where the identification is $\mathsf{H}_{univ}^{B}$-equivariant. We hence deduce that $S_2(U_0,\psi,\mathcal{M}(\mathsf{X},A))$ is a compact $\mathcal{O}$-module, since $\mathcal{O}$ is $p$-adically complete and each $S_2(U_n,\psi,A)$ is a finitely generated free $\mathcal{O}$-module. This allows us to define its ordinary part $S_2(U_0,\psi, \mathcal{M}(\mathsf{X},A))^{ord}$ as usual (see Section 2.4 of \cite{LongoVigni2012} and the references therein) as its direct summand on which the Hecke operator $T_p$ acts invertibly. We shorten the notation and denote by $\mathbb{W}$ the space $S_2(U_0,\psi, \mathcal{M}(\mathsf{X},\mathcal{O}))^{ord}$. In particular, the Hecke-equivariance in the inverse limit construction of $S_2(U_0,\psi, \mathcal{M}(\mathsf{X},\mathcal{O}))$, implies that $S_2(U_0,\psi, \mathcal{M}(\mathsf{X},\mathcal{O}))^{ord}=\varprojlim S_2(U_n,\psi,\mathcal{O})^{ord}$, where $T_p$ is replaced by $U_p$ on each component of the inverse limit. Proposition \ref{specialization maps prop} shows that the specialization maps descend to Hecke-equivariant specialization maps between the ordinary components,
	\begin{equation}
		\nu_{k,\varepsilon}^{ord}: \mathbb{W}\longrightarrow S_k(U_n,\psi\varepsilon, \mathcal{O}(\varepsilon))^{ord},
	\end{equation}
	with the same definition of $\nu_{k,\varepsilon}$ and where $\mathcal{O}(\varepsilon)$ is the finite extension of $\mathcal{O}$ generated by the values of $\varepsilon$.
	\begin{notation}
		We need to introduce some more notation. 
		\begin{enumerate}[label=(\alph*)]
			\item For any $m\geq 1$ and any character $\chi:\zp^\times\longrightarrow \overline{\qp}^\times$, let $\Psi_{m,\chi}:\mathsf{X}\longrightarrow \overline{\qp}^\times$ such that 
			\begin{equation}
				\Psi_{m,\chi}((x,y))=\begin{cases}
					\chi(y) & \textrm{ if } x\in p^m\zp,\\
					0 & \textrm{ otherwise.}
				\end{cases}
			\end{equation}
			In particular, $\Psi_{m,\chi}$ is homogeneous of degree $\chi$ for each $m$.
			\item Let $\Lambda=\mathcal{O}[\![ 1+p\zp]\!]$ be the extension of the classical Iwasawa algebra $\zp[\![ 1+p\zp]\!]$, obtained by $\mathcal{O}$. We set $\mathbb{W}_\Omega:=\mathbb{W}\otimes_{\Lambda}\Omega$ for any $\Lambda$-algebra $\Omega$.
			\item We say that a homomorphism $\kappa:\mathcal{R}\longrightarrow \overline{\qp}$ is an \emph{arithmetic homomorphism} if its restriction to $\zp^\times$ is of the form $\kappa_{|\zp^\times}(x)=x^{k-2}\varepsilon(x)$ for $k\geq 2$ and $\varepsilon:\zp^\times\longrightarrow \overline{\qp}^\times$ a character which factors through $\zp^\times/(1+p^n\zp)$, with $n$ minimal. In this situation, we say that $\kappa$ has weight $k$ and character $\varepsilon$ of conductor $p^n$.
		\end{enumerate}
	\end{notation}
	
	\begin{lemma}\label{Lemma inj}
		Let $\kappa:\mathcal{R}\longrightarrow \overline{\qp}$ be an arithmetic homomorphism  of weight $k$ and character $\varepsilon$ of conductor $p^n$. Let $F_\kappa$ be the field extension of $F$ containing the values of $\kappa$. The map $\nu_{k,\varepsilon}^{ord}$ induces the injective Hecke-equivariant homomorphism
		\begin{equation}
			\nu_{k,\varepsilon}^{ord}: \mathbb{W}_\mathcal{R}/\mathcal{P}_{\kappa}\mathbb{W}_\mathcal{R}\hookrightarrow S_k(U_n,\psi\varepsilon, F_\kappa)^{ord},
		\end{equation}
		where $\mathcal{P}_{\kappa}$ is the kernel of $\kappa$ in $\mathcal{R}$.
	\end{lemma}
	\begin{proof}
		We begin noting that $\mathcal{P}_{\kappa}S_2(U_0,\psi, \mathcal{M}(\mathsf{X},\mathcal{O}))=S_2(U_0,\psi, \mathcal{P}_{\kappa}\mathcal{M}(\mathsf{X},\mathcal{O}))$, as it can be seen by applying twice Lemma 1.2 of \cite{AshStevens1997} to $S_2(U_0,\psi, \mathcal{M}(\mathsf{X},\mathcal{O}))=H^{0}(U_0^p,H^1(F[\widehat{B}^{p,\times}],\mathcal{M}(\mathsf{X},\mathcal{O})))$. We prove now that $\mathcal{P}_{\kappa}\mathbb{W}=ker(\nu_{k,\varepsilon}^{ord})$. Let $\varphi\in \mathcal{P}_{\kappa}\mathbb{W}$; therefore $\varphi(\tilde{b})$ lies in $\mathcal{P}_{\kappa}\mathcal{M}(\mathsf{X},\mathcal{O})$ for any $\tilde{b}\in B^\times\backslash \widehat{B}^\times$. Lemma 6.3 of \cite{GreenbergStevens1993} shows that $\varphi(\tilde{b})\in \mathcal{P}_{\kappa}\mathcal{M}(\mathsf{X},\mathcal{O})$ if and only if $\varphi(\tilde{b})(f)=0$ for each homogeneous function of degree $\kappa$. For each $P\in L_{k-2}(F_\kappa)$, $\varepsilon(y)P(x,y)$ is homogeneous of degree $\kappa$ and hence $\nu_{k,\varepsilon}^{ord}(\varphi(\tilde{b}))(P)=0$ for each $\tilde{b}$ and $P$. Take now $\varphi\in ker(\nu_{k,\varepsilon}^{ord})$ and let $m\geq 1$. Since $T_p$ is invertible, let $\mu\in \mathbb{W}$ be such that $T_p^m\mu=\varphi$. Let $\gamma_a:=\mat{p}{a}{0}{1}$ for $a=0,\ldots,p-1$ and $\gamma_\infty:=\mat{1}{0}{0}{p}$. The definition of the Hecke operator $T_p$ in Section \ref{Hecke operators} does come from the coset decomposition $U_0\gamma_\infty U_0=\bigsqcup_{\alpha=0,\ldots,p-1,\infty}\gamma_\alpha U_0$. Then $T_p^m$ corresponds to a decomposition of the form $\bigsqcup_{i}\gamma_{m,i}U_0$, where each $\gamma_{m,i}$ is a product of $m$ matrices $\gamma_\alpha$, for $\alpha=0,\ldots,p-1,\infty$. We compute
		\begin{equation}
			\int_{p\zp\times\zp^\times}\Psi_{m,\kappa}(x,y)d(\varphi(\tilde{b}))(x,y)=\sum_{i}\int_{p\zp\times\zp^\times}\Psi_{m,\kappa}\left((x,y)\gamma_{m,i}\right)d(\mu\left(\tilde{b}\cdot \gamma_{m,i}\right))(x,y).
		\end{equation}
		For each $m$, $\Psi_{m,\kappa}\left((x,y)\gamma_{\infty}\right)=0$ thus, $\Psi_{m,\kappa}\left((x,y)\gamma_{m,i}\right)=0$ whenever $\gamma_{m,i}$ contains a copy of $\gamma_\infty$. Therefore, only the matrices $\gamma_{m,i}=\prod_{j=0,m-1}\gamma_{\alpha_{j}^i}=\mat{p^m}{\sum_j\alpha_{j}^ip^j}{0}{1}$ contribute to the integral and we recognize that 
		\begin{multline}
			\sum_{i}\int_{p\zp\times\zp^\times}\gamma_{m,i}\cdot \left(\varepsilon(y)y^{k-2}\right)d(\mu\left(\tilde{b}\cdot \gamma_{m,i}\right))(x,y)\\
			=U_p^m\int_{p\zp\times\zp^\times}\varepsilon(y)y^{k-2} d(\mu(\tilde{b}))(x,y)
			=U_p^m\nu_{k,\varepsilon}^{ord}(\mu(\tilde{b}))(y^{k-2}).
		\end{multline}
		By construction, $0=\nu_{k,\varepsilon}^{ord}(\varphi(\tilde{b}))=\nu_{k,\varepsilon}^{ord}(T_p^m\mu(\tilde{b}))=U_p^m\nu_{k,\varepsilon}^{ord}(\mu(\tilde{b}))$ and since $U_p$ is invertible on the space $S_k(U_n,\psi\varepsilon, F_\kappa)^{ord}$, $\nu_{k,\varepsilon}^{ord}(\mu(\tilde{b}))=0$ and hence $\nu_{k,\varepsilon}^{ord}(\varphi(\tilde{b}))(y^{k-2})=0$. Lemma 6.3 of \cite{GreenbergStevens1993} implies that $\varphi(\tilde{b})\in \mathcal{P}_{\kappa}\mathbb{W}$.
	\end{proof}
	\begin{remark}\label{freeness}
		As in the case of Eichler orders, the space of quaternionic modular forms $\mathscr{S}_k(U_n,\widetilde{\psi},\mathcal{O})$ is often finitely generated over $\zp$ and free, as it follows from the discussion in Section \ref{Quaternionic modular forms of weight $2$}. In particular, this holds true under the conditions discussed in Proposition \ref{torsion free Gammas}. Therefore, Lemma \ref{Lemma inj} implies that $\mathbb{W}_\mathcal{R}/\mathcal{P}_{\kappa}\mathbb{W}_\mathcal{R}$ is $\zp$-finitely generated and free. The discussion in Section \ref{Quaternionic modular forms of weight $2$} shows also that $S_2(U_0,\psi,\mathcal{M}(\mathsf{X},A))$ is often ${\widetilde{\Lambda}}$-free and finitely generated, once again, for example under the conditions in Proposition \ref{torsion free Gammas}. One can argue as in the proofs of Theorem 10.1, Corollary 10.3, and Corollary 10.4 of \cite{Hida1988b}, since the results proved there for quaternionic modular forms over definite quaternion algebras hold in more generality for special orders which are split at the interpolation prime $p$ (see also Remark \ref{remark on hida}).
	\end{remark}
	
	\subsection{The proof of the control theorem}
	
	As in Section \ref{choice of modular form}, we fix a $p$-ordinary newform $f$ in $S_2^{new}(\Gamma_1(Np^n\ell^{2r}),\psi,\mathbb{C})$, for $n\geq 1$ and with $\psi$ a Dirichlet character modulo $Np^n\ell^{2r}$ satisfying Assumptions \ref{chi-R} and \ref{psi hyp}. We also assume that its associated $p$-adic Galois representations is residually absolutely irreducible and $p$-distinguished. Let $f_\infty:\mathsf{H}_{\infty}^{B,ord}(F)\longrightarrow \mathcal{R}$ the homomorphism associated in Section \ref{choice of modular form} with the Hida family passing through $f$. Recall that $F=\qp(f)$ and that $\mathcal{O}$ is its ring of integers. For any $\mathcal{P}_\kappa$ as in the above Lemma \ref{Lemma inj}, we denote by $f^{B}_\kappa$ the composition
	\begin{equation}
		f^{B}_\kappa:\mathsf{H}_{univ}^B\longrightarrow \mathsf{H}^{B,ord}_\infty(F)\overset{f_\infty}{\longrightarrow} \mathcal{R}\longrightarrow \mathcal{R}_{\mathcal{P}_\kappa},
	\end{equation}
	where the first map is the compatible morphism of Section \ref{Hecke algebras} and the last map is the one to the localization of $\mathcal{R}$ at the prime $\mathcal{P}_\kappa$. We recall that $\mathcal{R}$ is a $\widetilde{\Lambda}$-algebra; we identify any $\widetilde{\Lambda}$-algebra as a $\Lambda$-algebra via the inclusion $\Lambda\hookrightarrow\widetilde{\Lambda}$ and write
	\begin{equation}
		\widetilde{\mathbb{W}}_{\kappa}:=\left(\mathbb{W}\otimes_{\Lambda} \mathcal{R}_{\mathcal{P}_\kappa}\right)[f^{B}_\kappa]
	\end{equation}
	for the isotypic component of the $\mathcal{R}_{\mathcal{P}_\kappa}$-module $\mathbb{W}\otimes_{\Lambda} \mathcal{R}_{\mathcal{P}_\kappa}$, where the Hecke operators act as determined by $f^{B}_\kappa$.
	\begin{proposition}\label{Prop inj}
		With the notation of Lemma \ref{Lemma inj}, there is an induced injective homomorphism
		\begin{equation}
			\nu_\kappa: \widetilde{\mathbb{W}}_{\kappa}/\mathcal{P}_\kappa \widetilde{\mathbb{W}}_{\kappa} \hookrightarrow \left(S_k(U_n,\psi\varepsilon, F_\kappa)^{ord}\right)[f_\kappa],
		\end{equation}
		for $f_\kappa$ the weight-$\kappa$ specialization of $f_\infty$.
	\end{proposition}
	\begin{proof}
		The proof is the same as of Proposition 3.5 in \cite{LongoVigni2012}, since it does not depend on the choice of the quaternionic order.
	\end{proof}
	As one can note from Theorem \ref{Thm 7.16-7}, the case of level $\ell^2$ and trivial character has to be handled with more care. The theory of Hida families for classical modular forms is well known and we can restrict our attention to the Hecke-submodules $S_k^{new}(\Gamma_1(Np^n\ell),\psi,F)^{ord}$ with $\psi$ a Dirichlet character modulo $Np^n$, with $n\geq 1$. We do not provide details here, but we refer to Chapter 7 of \cite{hida_1993} and Section 2 of \cite{LongoVigni2012}. We construct the space of ${\widetilde{\Lambda}}$-adic modular newforms, level $Np^n\ell$ and character $\psi$, as $\mathbb{W}^{\ell}:=\varprojlim S_2^{new}(\Gamma_1(Np^n\ell),\psi,F)^{ord}$. Moreover, we can twist its Hecke action by the character $\left(\frac{-}{\ell}\right)$ obtaining the corresponding space $\mathbb{W}^{\ell,\left(\frac{-}{\ell}\right)}:=\varprojlim S_2^{new}(\Gamma_1(Np^n\ell),\psi,F)^{\left(\frac{-}{\ell}\right),ord}$. As in Section \ref{Hecke algebras} we have an action of the universal Hecke algebra $\mathsf{H}_{univ}$ on $\mathbb{W}^{\ell,\left(\frac{-}{\ell}\right)}$.  In particular, taking $f$ in $S_k^{new}(\Gamma_1(Np^n\ell),\psi,F)^{\left(\frac{-}{\ell}\right),ord}$, the module $\left(\mathbb{W}^{\ell,\left(\frac{-}{\ell}\right)}\otimes_{\Lambda} \mathcal{R}_{\mathcal{P}_{(k,\varepsilon)}}\right)[f]$ is a free rank-1 $\mathcal{R}_{\mathcal{P}_{(k,\varepsilon)}}$-module (see Proposition 2.17 and the proof of Theorem 2.18 in \cite{LongoVigni2012}).
	\begin{lemma}\label{Main lemma}
		Assume $\mathbb{W}$ to be ${\widetilde{\Lambda}}$-free and finitely generated (see Remark \ref{freeness}). Suppose that $f\in T_2(n,r,\psi)$ and set, for any arithmetic homomorphism $\kappa=(k,\varepsilon)$,
		\begin{equation}
			{\mathbb{W}_{\kappa}}:=\begin{cases}
				\widetilde{\mathbb{W}}_{\kappa}\oplus \left(\mathbb{W}^{\ell,\left(\frac{-}{\ell}\right)}\otimes_{\Lambda} \mathcal{R}_{\mathcal{P}_{\kappa}}\right)[f_\kappa] & \textrm{ if $r=1$ and $\psi_\ell$ is the trivial character,}\\
				\widetilde{\mathbb{W}}_{\kappa} & \textrm{ otherwise,}
			\end{cases}
		\end{equation}
		where we let $\mathsf{H}_{univ}$ act on ${\mathbb{W}_\kappa}$ via the homomorphism $\mathsf{H}_{univ}\rightarrow \mathsf{H}_{univ}^B$ induced by the Jacquet--Langlands correspondence. Then ${\mathbb{W}_{\kappa}}$ is a free rank-2 $\mathcal{R}_{\mathcal{P}_\kappa}$-module.
	\end{lemma}
	\begin{proof}
		We start dealing with the case ${\mathbb{W}_{\kappa}}=\widetilde{\mathbb{W}}_{\kappa}$. We consider the $p$-divisible abelian group (\emph{cf.} Remark \ref{freeness} and Section \ref{Quaternionic modular forms of weight $2$}) $\mathbb{V} :=\varinjlim S_2^{ord}(U_n,\widetilde{\psi},F/\mathcal{O})$, where the inductive limit is taken with respect to the restriction maps induced by the inclusions $U_{n+1}\subset U_n$. The Hecke and diamond operators (at least away from $\ell$) act on $\mathbb{V}$ since, as in the case of Eichler orders, the restriction maps in \cite{Hida1988b} (see Eqs. (2.9a), (2.9b) and (3.5)) are compatible with the Hecke action. Taking the Pontryagin dual of $\mathbb{V}$ we obtain the Hecke-equivariant isomorphism $\widehat{\mathbb{V}}\cong\mathbb{W}$ (\emph{cf.} Eqs. \ref{shapiro} and \ref{inverse limit measures}), which shows it is a free ${\widetilde{\Lambda}}$-module of finite rank. We denote by $F_\kappa$ the field extension of $F$ generated by the values of $\varepsilon$ and by $\mathcal{O}_\kappa$ its ring of integers, which we can assume to be finite flat over $\zp$. Up to a scalar and up to taking the tensor product by $\mathcal{O}_\kappa$, we can suppose $f_\kappa$ to have coefficients in $\mathcal{O}$. We observe that $\widetilde{\mathbb{W}}_\kappa=\mathbb{W}[f_\kappa]\otimes_{\Lambda} \mathcal{R}_{\mathcal{P}_{\kappa}}$, as the action of the Hecke algebra is on the first component and the tensor product is just an extension of scalars. We can hence apply Theorem 9.4 of \cite{Hida1988b} (\emph{cf.} Remark \ref{remark on hida}) to $\mathbb{W}[f_\kappa]$ and obtain the isomorphism of $\mathsf{h}_n^{T,ord}(\mathcal{O})$-modules,
		\begin{equation}\label{iso 1 proof main lemma}
			\mathbb{W}[f_\kappa]\cong \widehat{\mathbb{V}[f_\kappa]}\cong S_k^{new}(U_n,\widetilde{\psi\varepsilon},\mathcal{O})[f_\kappa].
		\end{equation}
		We remark that the last Hecke-equivariant isomorphism in the above Eq. (\ref{iso 1 proof main lemma}) (as well as in Eq. (\ref{iso 2 proof main lemma})), comes from the restriction to $T_k(n,r,\psi\varepsilon)$ of the Pontryagin duality established in Lemma 7.1 of \cite{Hida1986b}; under the hypotheses of Lemma \ref{torsion free Gammas} one has the isomorphism $S_k(U_n,\widetilde{\psi\varepsilon},F/\mathcal{O})\cong S_k(U_n,\widetilde{\psi\varepsilon},\mathcal{O})\otimes F/\mathcal{O}$, as in the proof of Theorem 10.1 in \cite{Hida1988b}, and then Proposition \ref{square iso hecke-mf} recovers the needed Hecke-isomorphism for quaternionic modular forms. Similarly to the above discussion for $\mathbb{W}^{\ell}$, we can follow Section 2 of \cite{LongoVigni2012} and construct the interpolation module $\mathbb{W}^{\ell^{2r}}=\varprojlim_{n} S_2(\Gamma_1(Np^n\ell^{2r}),\psi,\mathcal{O})^{ord}$,  relative to the ordinary subspaces $S_k(\Gamma_1(Np^n\ell^{2r}),\psi\varepsilon, \mathcal{O})^{ord}$. We notice that under the hypothesis of Proposition \ref{torsion free Gammas}, the space $\mathbb{W}^{\ell^{2r}}$ is free of finite rank. In particular, we can reproduce the above chain of isomorphisms and obtain $\mathsf{h}_n^{T,ord}(\mathcal{O})$-isomorphisms
		\begin{equation}\label{iso 2 proof main lemma}
			\mathbb{W}^{\ell^{2r}}[f_\kappa]\cong S_k(\Gamma_1(Np^n\ell^{2r}),\psi\varepsilon)^{ord}[f_\kappa]\cong T_k(n,r,\psi\varepsilon)^{ord}[f_\kappa].
		\end{equation}
		Applying Propositions \ref{square iso hecke-mf} and \ref{pointwise dim}, we deduce the isomorphism of $\mathsf{h}_\infty^{T,ord}(\mathcal{O})$-modules,
		\begin{equation}
			\mathbb{W}[f_\kappa]\cong 2\mathbb{W}^{\ell^{2r}}[f_\kappa].
		\end{equation}
		Tensoring over $\Lambda$ with $\mathcal{R}_{\mathcal{P}_\kappa}$, we obtain the isomorphism of $\mathsf{h}_\infty^{T,ord}(\mathcal{O})\otimes_{\Lambda}\mathcal{R}_{\mathcal{P}_\kappa}$-modules, 
		\begin{equation}
			{\mathbb{W}_\kappa}\cong 2\left(\mathbb{W}^{\ell^{2r}}\otimes_{\Lambda}\mathcal{R}_{\mathcal{P}_\kappa}\right)[f_\kappa].
		\end{equation}
		As in the proof of Theorem 2.18 of \cite{LongoVigni2012}, Proposition 2.17 of \emph{loc.cit.} guarantees that $\left(\mathbb{W}^{\ell^{2r}}\otimes_{\Lambda}\mathcal{R}_{\mathcal{P}_\kappa}\right)[f_\kappa]$ is a free $\mathcal{R}_{\mathcal{P}_\kappa}$-module of rank 1, therefore ${\mathbb{W}_\kappa}$ is a free $\mathcal{R}_{\mathcal{P}_\kappa}$-module of rank 2.
		
		The case of $r=1$ and trivial character at $\ell$ is carried out similarly, once we define the $p$-divisible abelian group
		\begin{equation}
			\mathbb{V} :=\varinjlim \left( S_2^{new}(U_n,\widetilde{\psi},F/\mathcal{O})^{ord}\oplus S_k^{new}(\Gamma_1(Np^n\ell),\psi,F/\mathcal{O})^{\left(\frac{-}{\ell}\right),ord}\right),
		\end{equation}
		whose Pontryagin dual is $\mathbb{W}\oplus \mathbb{W}^{\ell,\left(\frac{-}{\ell}\right)}$.
	\end{proof}
	\begin{remark}\label{remark on hida}
		\begin{enumerate}[label=(\alph*)]
			\item The congruence subgroup we consider, away from $\ell$, is the one denoted by $V(Np^n)$ in \cite{Hida1988b} and one passes from this choice to the one used there by changing all the actions via $\mat{a}{b}{c}{d}\mapsto \mat{d}{-b}{-c}{a}$.
			\item We point out that Theorem 9.4 of \cite{Hida1988b} is stated under more strict hypotheses but, in the case of definite quaternion algebras, such hypotheses can be relaxed; this has been already noticed in \cite{LongoVigni2012} and \cite{Hsieh2021} in order to work with Eichler orders for algebras over $\q$, but Theorem 9.4 of \cite{Hida1988b} holds true also for special orders. This is due to the degree of generality in which the results of Chapter 8 of \cite{Hida1988b} are proved (as well as Lemma \ref{open compact order}), together with the necessity of a controlled behavior only at the interpolation prime $p$. Let us remark that we did not take into account the case of indefinite algebras, but that it seems to require a generalization of the spectral sequences approach contained in Chapter 9 of \cite{Hida1988b}.
		\end{enumerate}
	\end{remark}
	We can finally state the sought for Hida control theorem in the case of special orders of level $\ell^{2r}$.
	\begin{theorem}[Control theorem for special orders]\label{Control theorem for special orders}
		With the above notation, suppose $f$ to be twist-minimal at $\ell$. For any arithmetic homomorphism $\kappa:\mathcal{R}\longrightarrow \overline{\qp}$, the map $\nu_\kappa$ of Proposition \ref{Prop inj} induces an isomorphism of $2$-dimensional $F_\kappa$-vector spaces
		\begin{equation}
			\widetilde{\mathbb{W}}_{\kappa}/\mathcal{P}_\kappa \widetilde{\mathbb{W}}_{\kappa} \overset{\cong}{\longrightarrow} \left(S_k(U_n,\psi\varepsilon, F_\kappa)^{ord}\right)[f_\kappa].
		\end{equation}
		If $r=1$, $f$ has trivial character at $\ell$ and lies in $S_k^{new}(\Gamma_1(Np^n\ell),\psi,F_\kappa)^{\otimes\left(\frac{-}{\ell}\right)}$ (in particular, it is not twist-minimal at $\ell$), then the above isomorphism still holds, but the $F_\kappa$-vector spaces are $1$-dimensional.
	\end{theorem}
	\begin{proof}
		Suppose $f$ to be twist-minimal. Because of Propositions \ref{Prop inj} and \ref{pointwise dim} we know that
		\begin{equation}
			\dim_{F_\kappa}\left(\widetilde{\mathbb{W}}_{\kappa}/\mathcal{P}_\kappa \widetilde{\mathbb{W}}_{\kappa}\right)\leq 2
		\end{equation}
		and thus it is enough to prove the opposite inequality. Lemma \ref{Main lemma} shows that $\widetilde{\mathbb{W}}_\kappa$ is a free $\mathcal{R}_\kappa$-module of rank $2$. The case of $r=1$, trivial character at $\ell$ and $f\not\in S_k^{new}(\Gamma_1(Np^n\ell),\psi,F_\kappa)^{\otimes\left(\frac{-}{\ell}\right)}$ follows similarly. The remaining case accounts to the fact that the Jacquet--Langlands correspondence preserves twists.
	\end{proof}
	\noindent We can consider the finitely generated $\mathcal{R}$-module
	\begin{equation}
		\mathbb{W}_{\infty}:=\begin{cases}
			\left(\left(\mathbb{W}\oplus \mathbb{W}^{\ell,\left(\frac{-}{\ell}\right)}\right)\otimes_{\Lambda} \mathcal{R}\right)[f_\infty] & \textrm{ if $r=1$ and $\psi_\ell$ is the trivial character,}\\
			\left(\mathbb{W}\otimes_{\Lambda} \mathcal{R}\right)[f_\infty] & \textrm{ otherwise.}
		\end{cases}
	\end{equation}
	Proceeding similarly as in the proof of the above theorem we notice that $\mathbb{W}_\infty\otimes Frac(\Lambda)$ is a 2-dimensional $\mathcal{K}$-vector space, where $\mathcal{K}$ is the finite field extension of $Frac(\Lambda)$ called the \emph{primitive component} associated with the Hida family $f_\infty$ (see Section 3 in \cite{Hida1986} in particular, Theorem 3.5 and also Theorem 2.6a of \cite{GreenbergStevens1993}). As noticed in Section 2.2 of \cite{LongoVigni2012}, we point out that $\mathcal{R}$ is the integral closure of $\Lambda$ in $\mathcal{K}$. We can then formulate Theorem \ref{Control theorem for special orders} highlighting this global $\mathcal{R}$-module.
	\begin{theorem}\label{Control theorem for special orders 2}
		With the above notation, suppose $f$ to be twist-minimal at $\ell$. For any arithmetic homomorphism $\kappa:\mathcal{R}\longrightarrow \overline{\qp}$, the map $\nu_\kappa$ of Proposition \ref{Prop inj} induces an isomorphism of $2$-dimensional $F_\kappa$-vector spaces
		\begin{equation}
			\mathbb{W}_{\infty}\otimes_{\mathcal{R}} \mathcal{R}_{\mathcal{P}_\kappa}/\mathcal{P}_\kappa \mathcal{R}_{\mathcal{P}_\kappa} \overset{\cong}{\longrightarrow} \left(S_k(U_n,\psi\varepsilon, F_\kappa)^{ord}\right)[f_\kappa].
		\end{equation}
		If $r=1$, $f$ has trivial character at $\ell$ and lies in $S_k^{new}(\Gamma_1(Np^n\ell),\psi,F_\kappa)^{\otimes\left(\frac{-}{\ell}\right)}$ (in particular, it is not twist-minimal at $\ell$), then the isomorphism of $1$-dimensional $F_\kappa$-vector spaces holds:
		\begin{equation}
			\left(\left(\mathbb{W}\otimes_{\Lambda} \mathcal{R}\right)[f_\infty]\right)\otimes_{\mathcal{R}} \mathcal{R}_{\mathcal{P}_\kappa}/\mathcal{P}_\kappa \mathcal{R}_{\mathcal{P}_\kappa} \overset{\cong}{\longrightarrow} \left(S_k(U_n,\psi\varepsilon, F_\kappa)^{ord}\right)[f_\kappa].
		\end{equation}
	\end{theorem}
	\begin{corollary}
		Let $f_\infty$ be a primitive Hida family of tame level $N\ell^{2r}$, $r\geq 1$, tame character $\psi$ with its $\ell$-component, $\psi_\ell$, as in Assumption \ref{psi hyp}. Suppose moreover $f_\infty$ to be twist-minimal at $\ell$. Then there exist two $\mathcal{R}$-linearly independent elements $\phi^1_{f_\infty}$ and $\phi^2_{f_\infty}$ in $\left(\mathbb{W}\otimes_{\Lambda} \mathcal{R}\right)[f_\infty]$, which form a basis for $\left(\left(\mathbb{W}\otimes_{\Lambda} \mathcal{R}\right)[f_\infty]\right)\otimes \mathcal{K}$. Moreover, for any arithmetic homomorphism $\kappa$, $\nu_\kappa(\phi^1_{f_\infty})$ and $\nu_\kappa(\phi^2_{f_\infty})$ form a $F_\kappa$-basis for $\left(S_k(U_n,\psi\varepsilon, F_\kappa)^{ord}\right)[f_\kappa]$.
	\end{corollary}
	\begin{definition}
		We denote by $\mathcal{W}_{f_\infty}$ the $\mathcal{R}$-linear span of $\phi^1_{f_\infty}$ and $\phi^2_{f_\infty}$ and call it the \emph{subspace of special quaternionic Hida families associated with $f_\infty$}.
	\end{definition}
	
	
\subsection{A small remark on related works and open questions}\label{A small remark on related works and open questions}
	
	The mathematical literature about this situation of higher ramification at the primes at which the quaternion algebra ramifies is quite meager. Excluding the (singular and collective) works of Pizer, Hijikata and Shemanske, there are few other works considering special orders and they all share working with indefinite algebras. We already referred to \cite{LRdvP2018}, but we wish to point the reader's attention also to the two works \cite{Ciavarella2009} and \cite{deVeraPiquero2013}. In particular, in \cite{Pizer80p2}, Pizer defines certain local operators acting on the quaternionic modular forms. The present note leaves unanswered whether the two linearly independent quaternionic modular forms, and then the two Hida families, can be distinguished via some of these local operators. We wish to address carefully this question in the near future.
	
	{\setstretch{1.0}
		\bibliography{Bibliography}

\begin{thebibliography}{10}

\bibitem{AshStevens1997}
Anver Ash and Glenn Stevens.
\newblock p-adic deformations of cohomology classes of subgroups of
  {$GL(n,\mathbb{Z})$}.
\newblock {\em Collectanea Mathematica}, 48:1--30, 1997.

\bibitem{BDI2010}
Massimo Bertolini, Henri Darmon, and Adrian Iovita.
\newblock Families of automorphic forms on definite quaternion algebras and
  {T}eitelbaum's conjecture.
\newblock {\em Ast\'{e}risque}, (331):29--64, 2010.

\bibitem{Buzzard2004}
Kevin Buzzard.
\newblock {\em On p-adic Families of Automorphic Forms}, pages 23--44.
\newblock Birkh{\"a}user Basel, Basel, 2004.

\bibitem{Carayol1984}
Henri Carayol.
\newblock Repr\'esentations cuspidales du groupe lin\'eaire.
\newblock {\em Annales scientifiques de l'\'Ecole Normale Sup\'erieure}, 4e
  s{\'e}rie, 17(2):191--225, 1984.

\bibitem{Ciavarella2009}
Miriam Ciavarella.
\newblock Congruences between modular forms and related modules.
\newblock {\em Funct. Approx. Comment. Math.}, 41(part 1):55--70, 2009.

\bibitem{DallAva2021Approx-p}
Luca Dall'Ava.
\newblock Approximations of the balanced triple product {$p$}-adic
  {$L$}-function.
\newblock submitted, 2022.

\bibitem{DallAva2021PhD}
Luca Dall’Ava.
\newblock {\em Quaternionic {H}ida families and the triple product {$p$}-adic
  {$L$}-function}.
\newblock PhD thesis, Universt\"at Duisburg-Essen, 2021.
\newblock available at
  \url{https://duepublico2.uni-due.de/receive/duepublico_mods_00074866}.

\bibitem{deVeraPiquero2013}
Carlos de~Vera-Piquero.
\newblock The {S}himura covering of a {S}himura curve: automorphisms and
  \'{e}tale subcoverings.
\newblock {\em J. Number Theory}, 133(10):3500--3516, 2013.

\bibitem{GreenbergStevens1993}
Ralph Greenberg and Glenn Stevens.
\newblock {$p$}-adic {$L$}-functions and {$p$}-adic periods of modular forms.
\newblock {\em Inventiones mathematicae}, 111(2):407--448, 1993.

\bibitem{Gross87}
Benedict~H. Gross.
\newblock Heights and the special values of {$L$}-series.
\newblock In {\em Number theory ({M}ontreal, {Q}ue., 1985)}, volume~7 of {\em
  CMS Conf. Proc.}, pages 115--187. Amer. Math. Soc., Providence, RI, 1987.

\bibitem{Hida1986b}
Haruzo Hida.
\newblock Galois representations into
  $\mathrm{GL}_2(\mathbb{Z}_p[\![\mathrm{X}]\!])$ attached to ordinary cusp
  forms.
\newblock {\em Inventiones mathematicae}, 85:545--614, 1986.

\bibitem{Hida1986}
Haruzo Hida.
\newblock Iwasawa modules attached to congruences of cusp forms.
\newblock {\em Annales scientifiques de l'\'Ecole Normale Sup\'erieure}, Ser.
  4, 19(2):231--273, 1986.

\bibitem{Hida1988b}
Haruzo Hida.
\newblock On {$p$}-adic {H}ecke algebras for $\operatorname{GL}_2$ over totally
  real fields.
\newblock {\em Annals of Mathematics}, 128(2):295--384, 1988.

\bibitem{hida_1993}
Haruzo Hida.
\newblock {\em Elementary Theory of {L}-functions and {E}isenstein Series}.
\newblock London Mathematical Society Student Texts. Cambridge University
  Press, 1993.

\bibitem{HPS1989basis}
Hiroaki Hijikata, Arnold Pizer, and Thomas~R. Shemanske.
\newblock The basis problem for modular forms on {$\Gamma_0(N)$}.
\newblock {\em Mem. Amer. Math. Soc.}, 82(418):vi+159, 1989.

\bibitem{HPS1989orders}
Hiroaki Hijikata, Arnold Pizer, and Thomas~R. Shemanske.
\newblock Orders in quaternion algebras.
\newblock {\em Journal für die reine und angewandte Mathematik}, 394:59--106,
  1989.

\bibitem{Hsieh2021}
Ming-Lun Hsieh.
\newblock Hida families and {$p$}-adic triple product {$L$}-functions.
\newblock {\em Amer. J. Math.}, 143(2):411--532, 2021.

\bibitem{Li1975}
Wen Ch'ing~Winnie Li.
\newblock Newforms and functional equations.
\newblock {\em Math. Ann.}, 212:285--315, 1975.

\bibitem{LRdvP2018}
Matteo Longo, V\'{\i}ctor Rotger, and Carlos de~Vera-Piquero.
\newblock Heegner points on {H}ijikata-{P}izer-{S}hemanske curves and the
  {B}irch and {S}winnerton-{D}yer conjecture.
\newblock {\em Publ. Mat.}, 62(2):355--396, 2018.

\bibitem{LongoRotgerVigni2012}
Matteo Longo, Victor Rotger, and Stefano Vigni.
\newblock On rigid analytic uniformizations of {J}acobians of {S}himura curves.
\newblock {\em American Journal of Mathematics}, 134(5):1197--1246, 2012.

\bibitem{LongoVigni2012}
Matteo Longo and Stefano Vigni.
\newblock A note on control theorems for quaternionic {H}ida families of
  modular forms.
\newblock {\em Int. J. Number Theory}, 8(6):1425--1462, 2012.

\bibitem{MSwD1974}
Barry~C. Mazur and Peter Swinnerton-Dyer.
\newblock Arithmetic of {W}eil curves.
\newblock {\em Inventiones mathematicae}, 25(1):1--61, Mar 1974.

\bibitem{Miyake2006}
Toshitsune Miyake and Yoshitaka Maeda.
\newblock {\em Modular Forms}.
\newblock Springer Monographs in Mathematics. Springer Berlin Heidelberg, 2006.

\bibitem{P1977}
Arnold Pizer.
\newblock The action of the canonical involution on modular forms of weight
  {$2$} on {$\Gamma_{0}(M)$}.
\newblock {\em Math. Ann.}, 226(2):99--116, 1977.

\bibitem{Pizer1980}
Arnold Pizer.
\newblock An algorithm for computing modular forms on {$\Gamma_0(N)$}.
\newblock {\em Journal of Algebra}, 64(2):340 -- 390, 1980.

\bibitem{Pizer80p2}
Arnold Pizer.
\newblock Theta series and modular forms of level {$p^{2}M$}.
\newblock {\em Compositio Math.}, 40(2):177--241, 1980.

\bibitem{Shimizu1965}
Hideo Shimizu.
\newblock On zeta functions of quaternion algebras.
\newblock {\em Annals of Mathematics}, 81(1):166--193, 1965.

\end{thebibliography}
		\bibliographystyle{plain}}
	{   \hypersetup{hidelinks}	
		\Addresses}
	
\end{document}